\documentclass{amsart}
\usepackage[alphabetic,backrefs,lite,nobysame]{amsrefs} 
\usepackage{url,amssymb,enumerate,colonequals, bm}
\usepackage{color}
\usepackage[headings]{fullpage}
\usepackage{comment}
\usepackage[all]{xy}

\newcommand{\ord}{\mbox{ord}}

\usepackage{bookman}
\usepackage{hyperref}
\usepackage{calrsfs}
\usepackage{tikz-cd}

\theoremstyle{plain}

\newtheorem*{theorem*}{Theorem}
\newtheorem*{corollary*}{Corollary}
\newtheorem{theorem}{Theorem}[section]
\newtheorem{lemma}[theorem]{\bf Lemma}
\newtheorem{corollary}[theorem]{\bf Corollary}
\newtheorem{proposition}[theorem]{\bf Proposition}

\theoremstyle{definition}
\newtheorem{definition}[theorem]{\bf Definition}

\newtheorem{questions}[theorem]{\bf Questions}

\theoremstyle{remark}
\newtheorem{remark}[theorem]{\bf Remark}

\newtheorem{notationassumptions}[theorem]{\bf Notation and Assumptions}
\newtheorem{assumption}[theorem]{\bf Assumption}

\newcommand{\calM}{{\mathcal M}}

\newcommand{\calO}{{\mathcal O}}

\newcommand{\calS}{{\mathcal S}}

\newcommand{\calV}{{\mathcal V}}

\newcommand{\OO}{{\mathcal O}}

\newcommand{\C}{{\mathbb C}}
\newcommand{\F}{{\mathbb F}}

\newcommand{\N}{{\mathbb N}}
\newcommand{\Q}{{\mathbb Q}}

\newcommand{\Z}{{\mathbb Z}}

\newcommand{\pp}{{\mathfrak p}}

\newcommand{\qq}{{\mathfrak q}}

\newcommand{\cc}{{\mathfrak c}}
\newcommand{\ttt}{{\mathfrak t}}
\newcommand{\aaa}{{\mathfrak a}}
 \newcommand{\bb}{{\mathfrak b}}

\newcommand{\be}{\begin{enumerate}}
\newcommand{\ee}{\end{enumerate}}

\def\ord{\mathop{\mathrm{ord}}\nolimits}

\title[First-order definitions of rings of integral functions]{First-order definitions of rings of integral functions over algebraic extensions of function fields and undecidability}

\author{Alexandra Shlapentokh}
\address{Department of Mathematics, East Carolina University, Greenville, NC 27858}
\email{shlapentokha@ecu.edu}
\thanks{While working on this project AS was supported by an NSF FRG grant DMS-2152098}

\author{Caleb Springer}
\address{Center for Communications Research, Princeton, NJ 08540}
\email{c.springer.math@gmail.com}
\thanks{When this project began, CS was a Heilbronn Research Fellow at University College London, partially supported by the Additional Funding Programme for Mathematical Sciences, delivered by EPSRC (EP/V521917/1) and the Heilbronn Institute for Mathematical Research.}

\date{\today}

\begin{document}
\begin{abstract}
    In this paper, we study questions of definability and decidability for infinite algebraic extensions ${\bf K}$ of $\F_p(t)$ and their subrings of $\calS$-integral functions. 
    We focus on fields ${\bf K}$ satisfying a local property which we call $q$-\emph{boundedness}. This can be considered a function field analogue of prior work of the first author \cite{Shlapentokh18} which considered algebraic extensions of $\Q$.
    One simple consequence of our work states that if ${\bf K}$ is a $q$-bounded  extension of $\F_p(t)$, then for infinitely many non-constant $u$ the integral closure $\OO_{\bf K}$ of $\F_p[u]$ inside ${\bf K}$ is first-order definable in ${\bf K}$. Under the additional assumption that the constant subfield of ${\bf K}$ is infinite, it follows that both $\OO_{\bf K}$ and ${\bf K}$ have undecidable first-order theories, and that $\F_p[w]$ is definable in ${\bf K}$ for every non-constant $w$ in ${\bf K}$.  
    Our primary tools are norm equations and the Hasse Norm Principle, in the spirit of Rumely.  Our paper has an intersection with a recent arXiv preprint by Martinez-Ranero, Salcedo, and Utreras, although our definability results are more extensive and undecidability results are much stronger. 
\end{abstract}
\maketitle

\section{Introduction}
This paper investigates some old questions of definability and decidability in the language of rings over infinite algebraic extensions of global function fields.
Recall that a global function field is a finite extension of $\F_p(t)$ where $p$ is prime, $\F_p$ is a field with $p$ elements and $t$ is transcendental over $\F_p$.
Let ${\bf K}$ be a (possibly infinite) algebraic extension of $\F_p(t)$. 
Let ${\calO}_{\bf K}$ be the integral closure of $\F_p[t]$ in ${\bf K}$.
\begin{questions}
Consider the following questions of definability and decidability in the first-order language of rings possibly augmented with a finite number of constants.
\begin{enumerate}
    \item Is $\calO_{\bf K}$ definable over ${\bf K}$?
    \item Is the theory of ${\bf K}$ decidable?
    \item Is the theory of $\calO_{\bf K}$ decidable?
\end{enumerate}
\end{questions}
Before we proceed further let us note the following:
\begin{enumerate}
    \item The answers to the questions above are not independent of each other.  In particular, if $\calO_{\bf K}$ is definable over $\bf K$ and the theory of $\calO_{\bf K}$ is undecidable, then the theory of ${\bf K}$ is undecidable.  Conversely, if the theory of a ring of integral functions is decidable, then so is the theory of the field.
    \item For certain classes of fields ${\bf K}$ sufficiently ``close'' to the algebraic closure, the theory of the rings of integral functions of ${\bf K}$ and the theory of ${\bf K}$ itself are decidable; see \cites{GJR19, JR21, R23} for details.
    \item There is another approach to decidability and  definability in infinite extensions, namely considering these questions over a collection of fields rather than a single field 
    (see for example,  
    \cites{EMSW23, DF21} 
    and 
    \cite{FJ23}*{Theorems 23.9.3-4}).
\end{enumerate}

In this paper we show that there exists a large class of algebraic extensions ${\bf K}$ of $\F_q(t)$ such that rings of integral functions are definable within these fields.  
Moreover, in many cases, the rings of integral functions of these fields have an undecidable theory, in which case the definability result establishes that such fields have an undecidable first-order theory.
The class of fields ${\bf K}$ which we study is the class of \emph{$q$-bounded fields}. 
Here, $q$ is a rational prime number that is different from the characteristic of the field. 
The concept of $q$-boundedness is developed in full generality in Section~\ref{sec:q_bounded_def}, including a stronger condition called uniform $q$-boundedness.
In the case of Galois extensions of global function fields, $q$-boundedness and uniform $q$-boundedness are equivalent.
The following simplified condition implies both; see Proposition~\ref{prop:global_to_uniform_q}.

\begin{definition}[Global \texorpdfstring{$q$}{q}-boundedness]
\label{def:global_q_bounded}
Let $K/\F_p(t)$ be a finite extension and let ${\bf K}/K$ be a Galois algebraic extension.
Let $D=\sup\{\ord_q[\hat K: K]\}$ where the supremum is taken over all finite extensions $\hat K$ of  $K$ such that $\hat K \subset {\bf K}$.  Define ${\bf K}$ to be \emph{globally $q$-bounded} if $D<\infty$.
\end{definition}

Some of the consequences of our main results are the following theorems.
We restrict to the Galois case for simplicity here, and provide references to more general results in the main text. For the remainder of the section, let  $K/\F_p(t)$ be a finite extension and let ${\bf K}/K$ be a globally $q$-bounded Galois extension.

\begin{theorem}[Corollary~\ref{cor:S-ints_def}]
    The integral closure $\OO_{\bf K}$ of $\F_p[t]$ in ${\bf K}$ is first-order definable over ${\bf K}$.
\end{theorem}

To extend this definability result, we show that for any algebraic extension ${\bf K}$ of $\F_p(t)$, if one $\F_p$-polynomial ring $\F_p[u]$ is definable in ${\bf K}$, then every $\F_p$-polynomial ring is definable; see Theorem~\ref{thm:def}. Combining these results provides the following conclusion.

\begin{theorem}[Corollaries~\ref{cor:dec_field_infinite_constants_general} and \ref{cor:any_poly_ring_q_bounded}]
   If the algebraic closure of $\F_p$ in ${\bf K}$ is infinite,  then the first-order theory of ${\bf K}$ is undecidable. Moreover, if $w$ is any non-constant element of ${\bf K}$, then $\F_p[w]$ is first-order definable over ${\bf K}$ with parameters. 
\end{theorem}

En route to proving our main results, we also show the following definability result for integral closures of valuation rings, which may be of independent interest.

\begin{theorem}[Theorems~\ref{thm:val_ring_ex_def}-\ref{thm:val_ring_first_order_def}]
    If $\pp_K$ is a prime of $K$,
    then the integral closure $\OO_{\pp_K,{\bf K}}$ in ${\bf K}$ of the valuation ring of $\pp_K$ has a first-order definition over ${\bf K}$.
    Moreover, if the ramification degree of $\pp_K$ in ${\bf K}$ is finite, then there is an existential definition of $\OO_{\pp_K,{\bf K}}$.
\end{theorem}

We start with an overview of some relevant history of definability and decidability in Section~\ref{sec:history}.
In Section~\ref{sec:definability_and_norm_eqs}, we present some foundational results for viewing norm equations as polynomial equations.
The technical background for analyzing norms equations is presented in Section~\ref{sec:fields_and_norms}.
The first definability results appear in Section~\ref{sec:global}, which is dedicated to finite extensions of $\F_p(t)$.
We develop the concept of $q$-boundedness in Section~\ref{sec:q_bounded_def}, which is the key concept used in all subsequent sections.
The definability of integral closures of valuations rings is proven in Section~\ref{sec:def_val_ring}, which leads to the definability of rings of $\calS$-integers in Section~\ref{sec:q-bounded-S-ints}.
We deduce consequences of undecidability in Section~\ref{sec:decidability}.
Finally, in Section~\ref{sec:def_poly_rings}, we prove that if any $\F_p$-polynomial ring is definable in an algebraic extension of $\F_p(t)$, then every $\F_p$-polynomial ring is definable. 

\subsection{Valuations and primes of global function fields}
Valuations and primes of global function fields play a key role in the proofs of our results.  We briefly review the relevant definitions and suggest \cite{FJ23} for more details.
\begin{definition}
\label{def:val}
    Let $K$ be a global function field. Then a (non-trivial discrete) valuation $v_K$ of $K$ is a map $v_K: K \longrightarrow \Q \cup\{\infty\}$ with the following properties:
    \begin{enumerate}
    \item $v_K(K^{\times}) \cong \Z$ (for the sake of convenience we often ``renormalize'' $v_K(K^{\times})$ to be $\Z$),
\item $v(x)=\infty$ if and only if $x = 0$,
\item For all $x,y \in K$ we have $v(xy)=v(x)+v(y)$,
\item For all $x,y \in K$ we have $v(x+y)\geq \min(v(x),v(y))$.
\item For all $x,y \in K$ with $v(x)< v(y)$ we have that $v(x+y)=v(x)$.
    \end{enumerate}
\end{definition}

Let $\OO_v \subset K$ be the ring of all elements $x \in K$ with $v(x) \geq 0$.  One can show that $\OO_v$ is a local ring, that is $\OO_v$ has a unique maximal ideal consisting of all $x \in R_v$ with $v(x)>0$.  We will refer to this ideal as a prime of $K$ and denote it by $\pp_K$ (as well as other fraktur letters). 

The algebraic closure of $\F_p$ in $K$ will be referred to as the constant field of $K$. It is a finite extension of $\F_p$ because $K$ is a finite extension of $\F_p(t)$.  Any valuation of $K$ is trivial on its constant field, that is for any constant $c \in K$ we have that $v(c)=0$.  The degree of the residue field of $\pp_K$ over the constant field is called the degree of $\pp_K$.

The ideal $\pp_K$ of $\OO_v$ is generated by an element $\pi_K$, called the uniformizer, and each $x\in K^\times$ can be written as $x = u\pi_K^n$ for some $n\in \Z$ and $u\in \OO_v^\times$.
This number $n$ is called the order of $x$ at $\pp_K$ and is denoted by $\ord_{\pp_K}x$.
The ``normalization'' of the value group of $v_K$ as above results in $v_K(x)=\ord_{\pp_K}x$.  

All valuations of $\F_p(t)$ correspond to irreducible  polynomials in $t$ or to the degree of these polynomials.  Given a finite extension $L/K$ of global function fields, we can extend valuations of $K$ to $L$. The number of such extensions can be bigger than one but it is always finite.  If $v_L$ is an extension of $v_K$, then $[v_L(L):v_K(K)]$ is called the ramification degree of $v_L$ over $v_K$.  We will also denote the ramification degree by $e(\pp_L/\pp_K)$, where $\pp_L$ and $\pp_K$ are the prime ideals corresponding to $v_L$ and $v_K$ respectively. 
Indeed, if we extend $\pp_K$ to an ideal of the integral closure $\OO_{\pp_K, L}$ of $\OO_{\pp_K}$ inside $L$, then we recognize $\pp_L$ as a prime factor of $\pp_K\OO_{\pp_K, L}$. 
The degree of the residue field of $v_L$ over the residue field of $v_K$ is called the relative degree of $v_L$ over $v_K$ and is denoted by $f(\pp_L/\pp_K)$.  

If ${\bf K}$ is an algebraic (possibly infinite) extension of a global function field $K,$ a valuation $v_K$ has an extension $v_{\bf K}$ (possibly infinitely many extensions) to ${\bf K}$.  This extension $v_{\bf K}$ will satisfy all the properties from Definition \ref{def:val} except that $v_{\bf K}({\bf K}^{\times})$ may no longer be isomorphic to $\Z$. We can define the valuation ring $\OO_{v_{\bf K}}$, the unique maximal ideal $\pp_{\bf K}$ and its residue field as in the finite case.  

If $K_0=K \subset K_1 \subset \ldots$ is an infinite tower of global function fields such that ${\bf K}=\bigcup_{i=0}^{\infty}K_i$ and $v_{\bf K}$ is an extension of $v_K$ as above, then we can define $v_{K_i}$ to be the restriction $v_{{\bf K}|K_i}$ of $v_{{\bf K}}$ to $K_i$ and note that
\begin{itemize}
\item For each $i \geq0$, $v_{K_i}$ is a discrete valuation on $K_i$ extending $v_{K_{i-1}}$;
    \item $R_{\bf K}=\bigcup_{i=0}^{\infty} R_{v_{K_i}}$;
    \item $\pp_{\bf K}=\bigcup_{i=0}^{\infty}\pp_{K_i}$.
\end{itemize}
If we follow the factorization of $\pp_{K_0}$ through the tower $K_0\subseteq K_1\subseteq \dots$, we see that its factors in subfields of the tower form a tree.
We refer to this tree as the factor tree of $\pp_{K_0}$.
Observe that a choice of $\pp_{\bf K}$ over $\pp_{K_0}$ corresponds precisely to a choice of path through the factor tree.

\subsection{The role of \texorpdfstring{$q$}{q}-boundedness and norm equations}
As we mention below the existential definability of valuation rings of global function fields has been known for a long time due to results of Rumeley. He used the uniform nature of these definitions to define polynomial rings contained in these fields. The natural question that has arisen since then is whether these definitions can be extended to infinite algebraic extensions of global fields.  The notion of $q$-boundedness describes extensions where one can generalize some ideas of Rumeley to define integral closures of valuation rings and polynomial rings.  One can show that in a $q$-bounded field the Brauer group is non-trivial and while we do not use central simple algebras directly, we do use norm equations closely connected to these algebras.

The importance of $q$-boundedness can be seen from its influence on the solvability of norm equations.  
Let ${\bf K}$ be an infinite algebraic extension of a rational function field $\F_p(t)$.  
For a non-$q$-th power $a$ and a nonzero element $b$ in ${\bf K}$, consider the solvability of the following norm equation over ${\bf K}(\sqrt[q]{a})$.
\[
{\bf N}_{{\bf K}(\sqrt[q]{a})/{\bf K}}(y)=bx^q + b^q.
\]
As shown in Section~\ref{sec:definability_and_norm_eqs}, this norm equation corresponds to a system of polynomial equations over $\bf{K}$.
Note that we are interested in making sure that $x$ is integral at all but finitely many primes.
The key assumption of $q$-boundedness implies that for any ${\bf K}$-prime $\qq_{\bf K}$ there exists a finite extension $K \subset {\bf K}$ containing $\F_p(t)$ with the following property: for any finite extension $K'$ of $K$ inside ${\bf K}$, there is some $K'$-prime $\qq_{K'}$ below ${\qq_{\bf K}}$, some $a \in K'$ and some $b \in K'$ so that $\ord_{\qq_{K'}}b=-1$ and $\qq_{K'}$ is inert in the extension $K'(\sqrt[q]{a})/K'$. 
Then, if $x \in {\bf K}$ has a pole at $\qq_{\bf K}$, we find that $x$ has a pole at some $\qq_{K'}$ as above so that $\ord_{\pp_{K'}}(bx^q+b^q)<0$ and $\ord_{\qq_{K'}}(bx^q+b^q)\not \equiv 0 \bmod q$.  Since $\qq_{K'}$ is inert in the extension $K'(\sqrt[q]{a})/K'$, we can now conclude that the norm equation
\begin{equation}
\label{eq:normintro}
{\bf N}_{{K'}(\sqrt[q]{a})/{K'}}(y)=bx^q+b^q    
\end{equation}
has no solution locally at $\qq_{K'}$ and therefore no solution globally.  Further, for any sufficiently large finite extension $K''$ of $K'$ we can find $a$ and $b$ as above to block the solution to the corresponding norm equation over $K''$.  Thus, if we append $\forall a\forall b$ to our norm equation we will obtain a statement which does not hold over ${\bf K}$ if $x$ has a pole at $\pp_{\bf K}$.
We pause to remark that, in reality, we will choose $a$ and $b$ to be elements from certain large subsets, although this is not conceptually crucial for the exposition here.

Although this is a good start to producing a definition of $\mathcal S$-integral functions, more care is required because the norm equation \eqref{eq:normintro} can fail to have solutions for reasons other than poles of $x$.  To make sure that the poles of $x$ are the only obstruction to the solubility of the equation, we use a dynamic version of equation~\eqref{eq:normintro}.  In other words we change the fields where we try to solve the norm equation depending on $x$ and $a$.  In the modified field all zeros of $bx^q+b^q$ and all poles and zeros of $a$ are of order divisible by $q$.  Thus, adjoining $q$-th root of $a$ produces an unramified extension and we may apply the Hasse Norm Principle. In particular, zeros of $bx^q+b^q$ have no influence on the existence of solutions to the norm equations.

\begin{notationassumptions}
\label{notation}
We use the following throughout the paper.
\begin{itemize}
    \item We use $q$ and $p$ to denote distinct rational prime numbers.
    \item For $m\geq 1$, let $\F_{p^m}$ denote a finite field of $p^m$ elements.
    \item Let $\zeta_q$ be a primitive $q$-th root of unity in the algebraic closure of $\F_p$.
    \item Let $t$ be transcendental over $\F_p$.
    \item Let $K, L, M, \ldots$ denote finite extensions of $\F_p(t)$.
    \item Let ${\mathbf K}, {\mathbf L}, {\mathbf M}, \ldots$ denote (possibly infinite) algebraic extensions of $K, L, M, \ldots$.
    \item For ${\bf K}$ as above, let ${\bf k}$ denote the algebraic closure of a finite field in ${\bf K}$.  We will refer to ${\bf k}$ as the field of constants of ${\bf K}$.
    \item Let $\qq_K, \pp_K, \ldots$ denote primes or valuations of $K$. If $L/K$ is a finite extension, then let $\qq_L, \pp_L, \ldots$ denote primes of $L$ lying above $\qq_K, \pp_K,\ldots$.  In other words, we assume that the following equality of ideal holds:  $\pp_L \cap K=\pp_K$.
    \item For a prime $\pp_K$ (resp. a valuation $v_K$), let $\kappa_{\pp_K}$ (resp. $\kappa_{v_K}$) denote the residue field of $K$ at $\pp_K$.
    \item Given ${\bf K}$ and $\pp_{\bf K}$ we will also write $\pp_{\bf K}$ for the extension of $\pp_{\bf K}$ to the corresponding completion ${\bf K}_{\pp_{\bf K}}$.
    \item For an algebraic extension  ${\bf K}$ of a global function field $K$, write $\OO_{\pp_K}$ for the valuation ring of $\pp_K$ and let $\OO_{\pp_K,{\bf K}}$ be the integral closure of $\OO_{\pp_K}$ in ${\bf K}$.
    \item Let $\calS_K$ be a finite set of valuations of a global function field $K$.  If ${\bf K}$ is an algebraic extension of $K$, then let $\OO_{{\bf K}, \calS_{K}}$ be the integral closure of the ring of $\calS_K$-integers $\OO_{K,\calS_K}$ in ${\bf K}$. 
    We call $\OO_{{\bf K}, \calS_{K}}$ the ring of $\calS_K$-integral functions in ${\bf K}$.
    
    \item Assuming ${\bf K}$ and $\calS_K$ are as above, let $U_{{\bf K},\calS_K}$ be the set of elements of ${\bf K}$ that are units at all extensions of $\calS_K$ in ${\bf K}$ and let $\calV_{{\bf K}, \calS_K}$ be the set of elements $x$ in  ${\bf K}$ such that $v(x-1)>0$ for all $v$ extending valuations of $\calS_K$.
\end{itemize}
\end{notationassumptions}

\section{A brief history of the questions of definability and decidability of global fields and their algebraic extensions}
\label{sec:history}

The modern part of the first-order definability story starts with Julia Robinson who showed in 1949 that $\Z$ had a first-order definition over $\Q$ and therefore the first-order theory of $\Q$ was undecidable \cite{JRobinson49}.  Later she also showed that $\Z$ was first-order definable over rings of integers of number fields and rings of integers of number fields are first-order definable over their fields of fractions \cite{JRobinson59}.  Robinson also considered infinite algebraic extensions of $\Q$ and produced first examples of integral closures of $\Z$ in infinite algebraic extensions where $\Z$ was first-order definable (see \cite{JRobinson62}).  In the same paper she also constructed a uniform (i.e. independent of the number field in question) first-order definition of $\Z$ over rings of integers of number fields. 
In \cite{Rumely80}, R. Rumely showed that a polynomial ring is first-order definable over any global function field. His definition was uniform across all global function fields and required only one absolutely necessary parameter, i.e. the parameter that indicates which polynomial ring is being defined.
Both J. Robinson and R. Rumely produced existential definitions of valuation rings, over number fields in the case of J. Robinson and all global fields in the case of R. Rumely.

In \cite{Denef80} J. Denef, later joined by L. Lipshitz \cite{DL78}, initiated the study of existential definability and decidability over global fields and function fields of characteristic 0.  
Denef and Lipshitz showed that $\Z$ has an existential definition over the ring of integers of any extension of degree at most 2 of a totally real number field (\cite{DL78}), and therefore H10 is undecidable over those rings. This class of fields includes all abelian extensions of $\Q$, for example.
Over number fields they were followed by a number of other people, including Pheidas \cite{Pheidas88}, Shlapentokh \cite{Shlapentokh89}, and Videla \cite{Videla89}.  
In \cite{Poonen02}, Poonen tied the ranks of elliptic curves to the definability of $\Z$.  
His original idea was generalized in \cite{Shlapentokh08} and \cite{Shlapentokh09}, and eventually to a result of Mazur and Rubin (\cite{MR10}) showing that Shafarevich-Tate conjecture implied that $\Z$ has a Diophantine definition over rings of integers of all number fields.  
Pasten and Murty in \cite{MP18} showed that other well-known conjectures in Number Theory had the same implication.  
Using elliptic curves, Garcia-Fritz and Pasten  constructed a new family of number fields where $\Z$ was definable over rings of integers (\cite{GFP20}).  
There are also a number of other results concerning existential definability of $\Z$ over the rings of algebraic numbers of finite and infinite algebraic extensions of $\Q$ and bigger rings (\cite{Shlapentokh97}, \cite{Shlapentokh04}, \cite{PS05},\cite{MRS24}, \cite{CPZ05}, \cite{EGS11}, \cite{EE09}, \cite{EMPS17}, \cite{Miller22}).  
In particular, from the results of the first author and Kato we know that $\Z$ is existentially definable over the ring of integers of any abelian algebraic extension of $\Q$ with finitely many ramified primes \cite{Shlapentokh08}.

There are a number of definability and decidability results over rings and fields of functions of characteristic 0 (\cite{MS22}, \cite{Demeyer10}, \cite{DD12}, \cite{Eisentraeger04}, \cite{Eisentraeger07}, \cite{MB05}, \cite{MBS09}).  Of a particular interest to the readers of this paper may be a result of Kollar showing that $\C[t]$ is not existentially definable over $\C(t)$ (\cite{Kollar08}).

In 1991, Pheidas constructed an existential model of $\Z$ over global function fields of characteristic greater than 2; see \cite{Pheidas91}. In 1994, Videla did the same for global function fields of characteristic 2 in \cite{Videla94}.  These results were generalized by the first author and Eisentr\"ager in \cites{Shlapentokh96, Shlapentokh00, Shlapentokh02_Compositio, Eisentraeger03, Shlapentokh03, ES09,Eisentraeger12}.  In \cite{ES17}, they showed that any function field of positive characteristic has a first-order model of $\Z$ and therefore the first-order theory of any function field of positive characteristic is undecidable.  In the same paper it was shown that as long as a function field of positive characteristic did not contain the algebraic closure of a finite field, the field in question had an existential model of $\Z$ and therefore H10 was undecidable over such a field.

In 1992, generalizing results of Denef \cite{Denef79}, the first author showed that $\Z$ has an existential model over a ring of $\calS$-integers of any function field of positive characteristic \cite{Shlapentokh92}. 
The first author has also showed that rings of $\calS$-integers are existentially definable over big rings with the set of inverted primes of density arbitrarily close to 1 (see \cite{Shlapentokh98}, \cite{Shlapentokh02_Monatsh}).  
Eisentr\"{a}ger was the first person to consider defining valuation rings in infinite algebraic extensions of a rational field of positive characteristic, where the infinite extension was not obtained by extending the constant field (see \cite{Eisentraeger05}). 

In 2009, Poonen constructed a new first-order definition of finitely generated rings of algebraic integers over their fraction fields \cite{Poonen09}.  This definition was of the form $\forall\forall\exists \ldots \exists$.  This definition was structurally much simpler than the ones produces by J. Robinson and Rumely.   This definition was simplified further by K\"onigsmann who constructed a purely universal definition of $\Z$ over $\Q$ as well as a definition of the form $\forall\exists \ldots \exists$; see \cite{Koenigsmann16}.  J. Park showed the rings of algebraic integers of number fields also had a purely universal definition over their fields of fractions \cite{Park13}.  The problem of constructing a definition of the form $\forall\exists \ldots \exists$ for these rings remains open.

The number field results concerning new simpler first-order definitions over number fields led to similar results to over global function fields where the first author constructed a definition of the form $\forall\exists \ldots \exists$ for rings of $\calS$-integers \cite{Shlapentokh15} and Eisentr\"{a}ger and Morrison produced a purely universal definition of these rings \cite{EM18}.

In 1986, a paper of Rumely showing that the existential theory of the ring of all algebraic integers is decidable (\cite{Rumely86}) renewed interest in the study of the first-order theory over infinite algebraic extensions of global fields.  Of course the first decidability result concerning the first-order theory of a field belongs to Tarski who showed that this theory was decidable over the reals.  

Rumely's results were generalized by Van den Dries who showed that the first-order theory of the ring of all algebraic integers was decidable \cite{vandenDries88}.  It turned out that there were ``smaller'' rings of algebraic numbers where the theory was decidable and that a  similar decidability phenomena occurred over algebraic extensions of global function fields and was connected to instances where local-global principle applied. These results belong to, among others, to Green, Roquette, and Pop \cite{GPR95}, Moret-Bailly (\cite{MB89}), van den Dries and Macintyre (\cite{vdDM90}), Razon (\cite{R23}), Jarden and Razon (\cite{JR08}). 

Another remarkable result concerning decidability belonged to Fried, V\"olklein and Haran who showed that the first-order theory of the field of totally real algebraic numbers was decidable \cite{FHV93}.  We remind the reader that J. Robinson showed that the first-order theory of the ring of all totally real algebraic integers is undecidable \cite{JRobinson62}.  So this field constitutes a boundary of some sort between decidable and undecidable.  We are very far from understanding the exact nature of this boundary.  Some of the first steps in this directions were taken by the first author in \cite{Shlapentokh18} and the second author in \cite{Springer23}.  
A result on existential undecidability by the first author which is directly relevant for this paper states that the existential theory in the language of rings was undecidable over rings of $\calS$-integer of any function field of positive characteristic and in some cases over integral closures of rings of $\calS$-integers in infinite extensions of function fields of positive characteristic; see \cite{Shlapentokh92}.

Of course the results mentioned above do not constitute a complete history of first-order definability and decidability to date.  For more information on the matter we refer the reader to \cite{Darniere00}, \cite{Pasten23}, \cite{Poonen08}, \cite{ShlapentokhBook}.

\subsection{\texorpdfstring{A result of Martinez-Ranero, Salcedo, and Utreras (\cite{MRSU24})}{A result of Martinez-Ranero, Salcedo, and Utreras}}
A recent arXiv preprint by Martinez-Ranero, Salcedo, and Utreras explores questions similar the ones raised in this paper.  
There is an intersection with our results primarily due to the fact that the definability method used in both papers is based on the method used in \cite{Shlapentokh18}.  
Namely, both \cite{MRSU24} and this paper define rings of integral function over $q$-bounded infinite extensions of $\F_p(t)$. In   \cite{MRSU24} this definition requires at least one absolutely $q'$-bounded prime, where $q'$ is a rational prime number which maybe equal to $q$. No extra requirements are necessary for our definition.
In \cite{MRSU24}, the authors also require the extension to be separable, while we do not. 

Additionally, we construct a first-order definition of the constant field without any parameters and no assumptions about uniformly $q$-bounded primes (named \emph{completely} $q$-bounded in \cite{MRSU24}). In the case of an infinite constant field, we construct a definition (with a parameter) of a polynomial ring. Observe that at least one parameter is necessary to indicate which polynomial ring is being defined.
However, the main difference between the papers is the undecidability results.  
The main undecidability result from \cite{MRSU24} is stated below.
\begin{theorem}[Martinez-Ranero, Salcedo and Utreras, Theorem 5.1]
\label{thm:MRSU}
Let $p$ be an odd prime and $r$ a prime congruent
to $3 \bmod 4$ and such that $p$ is a quadratic residue modulo $r$.
The first-order theory of the structure $\F_{p^{a}}(t^{r^{-\infty}})$ in the language of rings $\{0, 1, +, \times\}$
interprets the theory R of Tarski, Mostowski and Robinson. In particular, it is undecidable.
\end{theorem}
This theorem relies on the following result.
\begin{theorem}[Martinez-Ranero, Salcedo and Utreras, Theorem 4.1]
\label{MRSU2}
There exist infinitely many prime numbers r such that, for every prime power $p^a$, the theory of the structure
\[
\F_{p^a}[t^{r^{-\infty}}]
=\bigcup_{n\in \Z_{>0}}
\F_{p^a}[t^{r^{-n}}]
\]
in the language $\{0, 1, t, +,\times\}$ interprets the Theory $R$ and is thus undecidable.
\end{theorem}
The fact that the structure
\[
\F_{p^a}[t^{r^{-\infty}}]
= \bigcup_{n\in \Z_{>0}}
\F_{p^a}[t^{r^{-n}}]
\]
in the language $\{0, 1, t, +, \times \}$ is undecidable is a corollary to a much more general result of the first author from (\cite{Shlapentokh92}*{Theorem~5.1}).  Actually, in \cite{Shlapentokh92} it is shown that the ring in Theorem \ref{MRSU2} is {\it existentially} undecidable for any positive characteristic $p$, including $p=2$. See Section \ref{subsec:fin_factors} for more details.

In our paper we show, among other undecidability results, that if a  $q$-bounded extension of $\F_p(t)$ has an infinite field of constants, then its first-order theory is undecidable; see Corollary~\ref{cor:dec_field_infinite_constants_general}.  We also show undecidability in the case of a $q$-bounded field with an inert or totally ramified  prime whether or not the field of constants is infinite; see Section~\ref{subsec:fin_factors}.

\section{Definability and norm equations}
\label{sec:definability_and_norm_eqs}

In this paper, we will pursue definability by using norm equations where the variables range over radical extensions of bounded degree. 
The following lemma clarifies that such equations correspond to existential sentences over the base field. 
This fact will be used implicitly throughout the remainder of the paper.

\begin{lemma}
\label{le:norm_as_diophantine}
Let $K$ be a field of characteristic prime to $q$ and fix a polynomial $P(t)\in K[t]$. Let $S$ be the set of tuples $(a,x,c_1,\dots, c_n)\in K^{n+2}$ such that 
\begin{equation}
    \exists y \in L: \ {\bf N}_{L(\sqrt[q]{a})/L}(y) = P(x),
    \label{eq:gen_norm_eq}
\end{equation}
where $L = K(\sqrt[q]{c_1}, \dots, \sqrt[q]{c_n})$. Then the set $S$ is existentially definable in $K^{n+2}$.
\end{lemma}
\begin{proof}
    Define a tower of fields $K_j := K(\sqrt[q]{c_1}, \dots, \sqrt[q]{c_j})$ for each $0\leq j \leq n$. 
    Note that this tower $K = K_0\subseteq K_1\subseteq\dots \subseteq K_n =L$ is formed via $n$ extensions generated by $q$-th roots, 
    and each step in the tower may or may not be a trivial extension. 
    By working recursively, it is clear that the property of whether $c_i$ is a $q$-th power over $K_{i-1}$ is diophantine over $K$.
    Hence, we may define $S$ via a disjunction of $2^n$ terms, where each term of the disjunction corresponds to a sequence of $\textsc{Yes}/\textsc{No}$ answers to whether the extension $K_i/K_{i-1}$ is trivial for each $1\leq i\leq n$.
    Therefore, it is enough to assume that $c_i$ is a non-$q$-th power in $K_{i-1}$ for all $i$.  In this case, $B_{L/K} = \{\prod_{i} \sqrt[q]{c_i}^{v_i} : \vec v\in \{0,\dots, q-1\}^n\}\}$ is a $K$-basis for $L$.

   We remark that the meaning of the norm equation \eqref{eq:gen_norm_eq} depends on whether $\sqrt[q]{a}$ is a $q$-th power in the base field $L$ or not; cf. \cite{Lang_Algebra}*{Theorem~VI.9.1}.
    In the former case, the equation is trivially satisfied because $L = L(\sqrt[q]{a})$ and ${\bf N}_{L/L}(P(x)) = P(x)$ for all $x\in L$. Therefore, we will form another disjunction where the first term is $\exists \alpha \in L: \ \alpha^q = a$. This formula can be reinterpreted as a system of equations over $K$ by rewriting $\alpha$ in terms of the basis $B_{L/K}$, expanding the product and comparing terms. 
    
    It only remains to represent the norm equation in the case where $L(\sqrt[q]{a})/L$ is a non-trivial extension, so we assume that $a$ is not a $q$-th power in $L$.
    Any $y\in L(\sqrt[q]{a})$ is written in the natural 
    $L$-basis as $y = d_0 + d_1\sqrt[q]{a} + \dots + d_{q-1}\sqrt[q]{a}^{q-1}$, 
    and its norm is
    $$
        \prod_{j = 0}^{q-1} \left(d_0 + d_1\zeta_q^j\sqrt[q]{a} + \dots + d_{q-1} \zeta_q^{j(q-1)}\sqrt[q]{a}^{q-1}\right)
    $$
    where $\zeta_q$ is a primitive $q$-th root of unity.
    In particular, the norm is clearly a polynomial over $L$ in $a$ and the coefficients of $y$ in the natural $L$-basis $\{1, \sqrt[q]{a}, \dots, \sqrt[q]{a}^{q-1}\}$. 
    Using this polynomial, expanding each element of $L$ in \eqref{eq:gen_norm_eq} in terms of the basis $B_{L/K}$ provides the desired formula over $K$.
\end{proof}

When we apply this lemma in practice, it is helpful for the base field $K$ to contain a $q$-th root of unity, since this implies that the extension $L(\sqrt[q]{a})/L$ appearing in \eqref{eq:gen_norm_eq} is cyclic; see \cite{Lang_Algebra}*{Theorem~VI.6.2}.
For the sake of definability, this assumption on $K$ can be made without loss of generality, as shown in the following proposition and corollary.
The method of proof involves essentially the same trick of rewriting equations, although we note that the following proposition has a fixed field extension, whereas the field extension varies in \eqref{eq:gen_norm_eq}.

\begin{proposition}
\label{prop:fo}
    Let $H/F$ be a finite extension of fields. If $X \subset H$ is a first-order definable set, then $X \cap F$ is a first-order definable subset of $F$.  Moreover, if $X$ is existentially definable over $H$, then $X\cap F$ is existentially definable over $F$.
\end{proposition}
\begin{remark}
    The existential version of this proposition is well-known.  See, for example \cite{ShlapentokhBook}[Lemma 2.1.17]. 
\end{remark}
\begin{proof}

Let $\{\alpha_1,\dots,\alpha_r\}$ be a basis for $H$ over $F$, which is finite by hypothesis. 
By definition, this means that every element of $H$ is represented uniquely as a sum $\sum_{j=1}^{r} c_j\alpha_j$ when the coefficients $c_j$ range over elements of $F$. 
We will exploit a well-established method of rewriting equations with respect to this basis.

Assume $X\subseteq H$ is defined by the formula
\begin{equation}
\label{eq:inH}
E_1y_1 \in H\ldots E_sy_s \in H \ P(x,\bar y)=0,
\end{equation}
where  $P(x,\bar y) \in H[x,\bar y]$ and each $E_i$ is either a universal or existential quantifier.  

For each $1\leq i\leq s$, we replace the variable $y_i$ in the polynomial $P(x,\bar y)$ with the sum $\sum_{j=1}^{r}y_{i,j}\alpha_j$.
By expanding the polynomial, and rewriting all of the $H$-coefficients in terms of the basis $\{\alpha_1,\dots, \alpha_r\}$ as needed, we may write
\[
P(x,\sum_{j=1}^{r}y_{1,j}\alpha_j, \ldots, \sum_{j=1}^{r}y_{s,j}\alpha^j)
 =\sum_{j = 1}^r \alpha_j Q_j(x,y_{1,1}, \dots, y_{s,r}).
\]
where each $Q_j$ is a polynomial with coefficients in $F$.
Because $\{\alpha_1,\dots, \alpha_r\}$ is an $F$-basis for $H$, we observe that \eqref{eq:inH} is equivalent to
\begin{equation}
\label{eq:inH2}
E_1y_{1,1}\in F \dots E_1 y_{1,r}\in F \dots E_s y_{s,r} \in F \ Q_1(x,y_{1,1}, \dots, y_{s,r})=\dots = Q_r(x,y_{1,1}, \dots, y_{s,r})=0,
\end{equation}
This formula gives the first-order definition of $X\cap F$ in $F$. Moreover, we observe that precisely the same quantifiers are used, and in the same order, except that each quantifier is now used $r$ times in a row.
\end{proof}

From Proposition \ref{prop:fo} we get the following corollary justifying the assumption following it.
\begin{corollary}
    \label{cor:zeta_in_K}
    Let $F$ be a field of characteristic not equal to $q$ and let $\zeta_q$ be a primitive $q$-th root of unity.
    If a subset $A\subseteq F(\zeta_q)$ is definable over $F(\zeta_q)$, then $A\cap F$ is definable over $F$. 
\end{corollary}
\begin{assumption}
    \label{as:zeta_in_K}
    From now on we assume that $K$ (and therefore ${\bf K}$) contains a primitive $q$-th root of unity and therefore any cyclic extension of degree $q$ is generated by taking a $q$-th root of an element of $K$ (or ${\bf K}$).
\end{assumption}

Recall that Lemma~\ref{le:norm_as_diophantine} defines a diophantine subset of tuples. 
In the material that follows, we will combine this fact with the following lemma to obtain our main results.
\begin{lemma}
\label{le:fo}
    Let $F$ be a field, let $A \subset F^3, B\subset F, C\subset F$ be diophantine subsets of $F^3$ and $F$ respectively.  Let 
    \[
    D=\{x \in F\mid\forall x_1 \in B \forall x_2 \in C: (x_1, x_2,x) \in A\}.
    \]
    Then $D$ is first-order definable over $F$.
\end{lemma}
\begin{proof}
    Observe that we may write $D$ in the following form:
    $$
    D=\{x \in F\mid \forall x_1,x_2 \in F:((x_1 \not \in B)\lor (x_2 \not \in C) \lor ((x_1,x_2, x ) \in A)\}.
    $$
\end{proof}

\section{Field extensions and the solubility of norm equations}
\label{sec:fields_and_norms}

In this section, we provide the necessary foundational results for analyzing the solubility of norm equations over global function fields and their infinite algebraic extensions.
We choose to work with norm equations and radical extensions because we can exploit the local-global principle for cyclic extensions known as the \emph{Hasse Norm Principle}, which is recalled below. 
Indeed, the cyclic extensions of a field $L$ of degree $q$ are precisely the extensions $L(\sqrt[q]{a})/L$ generated by the $q$-th root of an element $a\in L\setminus L^q$, as long as $L$ is a field containing a primitive $q$-th root of unity and $q$ is distinct from the characteristic of $L$;  see \cite{Lang_Algebra}*{Theorem~VI.6.2}.
In light of Corollary~\ref{cor:zeta_in_K}, we maintain Assumption~\ref{as:zeta_in_K} without loss of generality.
In particular, we assume every field contains a primitive $q$-th root of unity, so every extension generated by a $q$-th root is cyclic.

\begin{theorem}[Hasse Norm Principle, \S9, \cite{Tate65}]
\label{thm:HNP}
Let $L/K$ be a cyclic extension of global function fields.  Let $x \in K$ and let $\calM_L$ be the collection of all valuations of $L$.  Then there exists $y \in L$ such that ${\bf N}_{L/K}(y)=x$ if and only if for all $v \in \calM_L$ we have that there exists $y \in L_v$ such that ${\bf N}_{L_v/K_v}(y)=x$.
\end{theorem}

The Hasse Norm Principle implies that the solubility of a norm equation reduces to inspecting the same equation over local fields.
These local equations, in turn, can be understood with the following well-known result.

\begin{proposition}
\label{prop:localunit}
Let $L,K$ and $v$ be as above and assume $L_v/K_v$ is unramified.  
If~$x\in K_v$ satisfies $\ord_v(x) \equiv 0 \bmod [L_v:K_v]$, then there exists $y \in L_v$ such that ${\bf N}_{L_v/K_v}(y)=x$.
\end{proposition}
\begin{proof}
    Let $\pi_{K_v}\in K_v$ be a uniformizer, and write $x = \pi_{K_v}^{\ord_v{x}}x_0$ where $x_0$ is a unit at $v$ and $\ord_v(x) = n[L_v:K_v]$.
    By \cite{Lang94}*{Lemma IX.3.4}, there is $y_0$ in $L_v$ so that ${\bf N}_{L_v/K_v}(y_0) = x_0$. Thus, $x$ is a norm as desired: ${\bf N}_{L_v/K_v}(\pi_{K_v}^ny_0) = \pi_{K_v}^{n[L_v:K_v]}x_0 = x$.
\end{proof}

We will apply these results to analyze equations of the form seen in Equation~\eqref{eq:gen_norm_eq}.
First, we provide a lemma describing the behavior of primes in such $q$-th root extensions. 
This allows us to deduce conditions where norm equations are or are not solvable.
\begin{lemma}
\label{le:split1}
Let $L$ be a global function field of characteristic $p >0$ which contains a primitive $q$-th root of unity where $q \ne p$ is a prime number.
Given $a \in L\setminus L^q$, the following statements are true concerning the extension $L(\sqrt[q]{a})/L$.
\begin{enumerate}[(a)]
    \item \label{it:1} $L(\sqrt[q]{a})/L$ is a Galois extension of degree $q$.
    \item \label{it:2} If $\ord_{\pp_L}a=0$ then $\pp_L$ splits (completely) in $L(\sqrt[q]{a})$ if there exists $u \in L$ such that $\ord_{\pp_L}(a-u^q)>0$, and is inert otherwise.
    \item \label{it:3} If $\pp_L$ is a prime of $L$, then $\pp_L$ is ramified if and only if $\ord_{\pp_L}a \not \equiv 0 \bmod q$.
\end{enumerate}
\end{lemma}
\begin{proof}
\eqref{it:1} In $L(\sqrt[q]{a})$, the polynomial $T^q-a$ splits completely because $L$ contains a primitive $q$-th root of unity.  Therefore, $L(\sqrt[q]{a})/L$ is Galois.

    \eqref{it:2}.  If $\pp_L$ splits completely in the extension $L(\sqrt[q]{a})/L$, then the polynomial $T^q-a$ has $q$ roots in $L_{\pp_L}$, that is $a$ is a $q$-th power in $L_{\pp_L}$ and also in the residue field  $k_{\pp_L}$ of $\pp_L$.  
    Thus, for any lift $u\in L$ of a root of $T^q-a$ in the residue field, we have $a - u^q\equiv 0 \bmod \pp_L$. 
    So, we must have $\ord_{\pp_L}(a-u^q)>0$.  
    Conversely, if there is no $u \in L$ such that $\ord_{\pp_L}(a-u^q)>0$, then  there is no $u \in L_{\pp_L}$ such that $u^q=a$.  
    Hence the polynomial $T^q-a$ remains prime in $L_{\pp_L}$ and the degree of the local extension at $\pp_L$ is $q$.  
    In other words, $\pp_L$ is inert in the extension $L(\sqrt[q]{a})/L$.

\eqref{it:3} 
Let $\pp_{L(\sqrt[q]{a})}$ be a prime of $L(\sqrt[q]{a})$ above $\pp_L$.  
Observe that 
\[
    e(\pp_{L(\sqrt[q]{a})}/\pp_L)\ord_{\pp_L}a
        =\ord_{\pp_{L(\sqrt[q]{a})}}a
        =q\ord_{\pp_{L(\sqrt[q]{a})}}\sqrt[q]{a}.
\]
Therefore, if $\ord_{\pp_L}a \not \equiv 0 \bmod q$, we have that $e(\pp_{L(\sqrt[q]{a})}/\pp_L) \equiv 0 \bmod q$ and $\pp_L$ is totally ramified in the extension $L(\sqrt[q]{a})/L$.
Conversely, if $\ord_{\pp_L}a \equiv 0 \bmod q$, then we define $w = \pi^{\frac{1}{q}\ord_{\pp_L}a}$ where $\pi$ is a local uniformizing parameter of $L$ at $\pp_L$. 
We deduce that $q\ord_{\pp_L}w=\ord_{\pp_L}a$  and $L(\sqrt[q]{a})=L(\sqrt[q]{aw^{-q}})$.
Then the element under the $q$-th root has order 0 at $\pp_K$.
Now by Part \eqref{it:2}, it follows that $\pp_K$ does not ramify in the extension $L(\sqrt[q]{a})/L$.
\end{proof}

Now we present two results which establish scenarios when norm equations do or do not have solutions, respectively.
First, we provide the key application of the Hasse Norm Principle for the purposes of our paper.

\begin{proposition}
\label{prop:hasse_application}
Let $L$ be a global function field of characteristic $p >0$ which contains a primitive $q$-th root of unity where $q \ne p$ is a prime number.
    Let $a \in L$  
    be an element whose divisor over $L$ is the $q$-th power of another $L$-divisor.  
    Let $z \in L$ be such that for every prime $\pp_L$ of $L$ one of the following holds:
    \begin{enumerate}[(a)]
        \item $\pp_L$ splits completely in the extension $L(\sqrt[q]{a})/L$;\label{it:hasse_app_split}
        \item $\ord_{\pp_L}z \equiv 0 \bmod q$.\label{it:hasse_app_0_q}
    \end{enumerate} 
    Under these assumptions,  $z$ is an $L(\sqrt[q]{a})$-norm in $L$.
\end{proposition}
\begin{proof}
    If $a$ is a $q$-th power, then the extension is trivial and the claim is automatically satisfied because ${\bf N}_{L/L}(z)=z$. Thus, we assume that $a$ is not a $q$-th power, which means that the extension is cyclic by Lemma~\ref{le:split1}.\eqref{it:1}.

    By the Hasse Norm Principle (Theorem~\ref{thm:HNP}), the norm equation ${\bf N}_{L(\sqrt[q]{a})/L}(y)=z$ has a solution $y\in L(\sqrt[q]{a})$ if and only if the equation is everywhere locally solvable. We obtain solubility locally everywhere by applying Proposition~\ref{prop:localunit}, as follows.
    
  Fix a prime $\pp_L$ of $L$ and write $M = L(\sqrt[q]{a})$. Since the divisor of $a$ is the $q$-th power of another $L$-divisor, the extension $M/L$ is unramified by Lemma \ref{le:split1}.\eqref{it:3}.  
  Let $\pp_{M}$ be a prime of $M$ lying over the prime $\pp_L$ of $L$. 
  Note that $[M_{\pp_{M}} : L_{\pp_L}] = q$ unless $\pp_L$ is split completely in $M/L$, in which case $[M_{\pp_{M}} : L_{\pp_L}] = 1$. 
  Therefore, $\ord_{\pp_L}z \equiv 0 \bmod [M_{\pp_{M}} : L_{\pp_L}]$ under either of the hypotheses on $z$, and consequently the norm equation is locally solvable at $\pp_L$ by Proposition~\ref{prop:localunit}.
\end{proof}

As a complement to the preceding proposition, we have the following simple observation. 
In practice, the condition that the prime is inert can be verified with the Lemma~\ref{le:split1}.\eqref{it:2}.

\begin{lemma}
\label{le:not_norm}
Let $L$ be a global function field of characteristic $p >0$ which contains a primitive $q$-th root of unity where $q \ne p$ is a prime number.
Fix $a\in L\setminus L^q$ and $z\in L$. 
Let $\pp_L$ be a prime of $L$ which is inert in the extension $L(\sqrt[q]{a})/L$.
If $\ord_{\pp_L}z \not \equiv 0 \bmod q$, then $z$ is not an $L(\sqrt[q]{a})$-norm in $L$.  
\end{lemma}
\begin{proof}
If $\pp_{L(\sqrt[q]{a})}$ lies above $\pp_L$ in $L$, then $f(\pp_{L(\sqrt[q]{a})}/\pp_L)=q$.  Thus, if $z \in L$ is an $L(\sqrt[q]{a})$-norm, we have  $\ord_{\pp_L}z \equiv 0 \bmod q$.
\end{proof}

We now turn our attention to infinite algebraic extensions of global function fields.
In order to apply methods for global fields to these larger fields, we require the following two results. 
The first is a straightforward observation which implies that any norm equation which is solvable over a global field will continue to be solvable over any further extension, including infinite extensions. 
The second result is more delicate, but can be intuitively viewed as stating that non-solvability of norm equations in an infinite degree extension is a phenomenon which can be observed in a certain finite subextension.

\begin{lemma}
\label{le:sol_down_to_sol_up}
    Let ${\bf L}/E$ be a (possibly infinite) extension of fields which are algebraic over $\F_p(t)$, and fix $a\in E$. For any $z\in E(\sqrt[q]{a})$, there is $z'\in {\bf L}(\sqrt[q]{a})$ so that 
    $$
        {\bf N}_{E(\sqrt[q]{a})/E}(z) = {\bf N}_{{\bf L}(\sqrt[q]{a})/{\bf L}}(z').
    $$
\end{lemma}
\begin{proof}
    There are two cases. If $\sqrt[q]{a}\not\in {\bf L}$, then the conjugates of $\sqrt[q]{a}$ over ${\bf L}$ are the same as the conjugates over $E$; see \cite{Lang_Algebra}*{Theorem~VI.9.1}. In particular, the claim follows after choosing $z' = z$, since the norm of $z$ (with respect to either extension) is simply expressed in terms of the conjugates of $\sqrt[q]{a}$.

    Otherwise, if $\sqrt[q]{a}\in {\bf L}$, then we have a trivial extension ${\bf L}(\sqrt[q]{a})={\bf L}$ and therefore ${\bf N}_{{\bf L}(\sqrt[q]{a})/{\bf L}}(z') =z'$ for all $z'$. Thus, we may take $z' ={\bf N}_{E(\sqrt[q]{a})/E}(z)$.
\end{proof}

\begin{proposition}
    \label{prop:finandinfin}
Consider a tower of algebraic extensions $\F_p(t)\subseteq E \subseteq {\bf K}\subseteq {\bf L}$ where ${\bf K}/E$ is infinite and  ${\bf L}/{\bf K}$ is generated over ${\bf K}$ by elements $\beta_1,\ldots, \beta_s \in {\bf L}$.
If $\alpha$ is algebraic over ${\bf L}$ and $z \in {\bf L}(\alpha)$, then there are finite extensions $\hat E/E$ and $L/\hat E$ satisfying the following conditions.
\begin{enumerate}
    \item $\hat E \subset {\bf K}$,
    \item $L \subset {\bf L}$,
    \item  $L={\hat E}(\beta_1,\dots,\beta_s)$ and $[{\hat E}(\beta_j):\hat E]=[{\bf K}(\beta_j):{\bf K}]$ for all $j$,
    \item   $\alpha$ has the same monic irreducible polynomial over ${\bf L}$ as over $L$.
    \item There is an equality of norms
\begin{equation*}
    {\mathbf N}_{{\bf L(\alpha)}/{\bf L}}(z)
    ={\mathbf N}_{{ L(\alpha)}/{ L}}(z).
\end{equation*}
\end{enumerate}

\end{proposition}
Before starting the proof, we present a diagram of the fields used in Propostion~\ref{prop:finandinfin}.
$$
\begin{tabular}{c}
\xymatrix{ 
{L(\alpha)} \ar[r]^{\text{infinite}} & {\bf L}(\alpha) \\
{L} \ar[u]^{\text{finite}}\ar[r]^{\text{infinite}} & {\bf L}\ar[u]_{\text{finite}} \\
{\hat E} \ar[u]^{\text{finite}} \ar[r]^{\text{infinite}} & {\bf K}\ar[u]_{\text{finite}}\\
{E} \ar[u]^{\text{finite}} \ar[ur]_{\text{infinite}} & 
             }\\
\end{tabular}
$$
\begin{proof}
Let $r_j:=[{\bf K(\beta_j)}:{\bf K}]$ and define $r = r_1\dots r_s$. Let $d_{0,j}+d_{1,j}T +\ldots + T^{r_j}, d_{i,j} \in {\bf K}$ be the monic irreducible polynomial of $\beta_j, j=1, \ldots,s$ over ${\bf K}$.
Let $\ell:=[{\bf L}(\alpha):{\bf L}]$, let $g_0+g_1T +\ldots +T^{\ell}, g_i \in {\bf L}$ be the monic irreducible polynomial of $\alpha$ over ${\bf L}$.
Let $\omega_i=\prod_{j=1}^s\beta_j^{a_{i,j}}$, where $i=1,\ldots,r$, $a_{i,j} \in \{0,\ldots,r_j-1\}$ and all $s$-tuples $(a_{i,1},\ldots,a_{i,s})$ are distinct.
Then $\Omega=\{\omega_1,\ldots,\omega_{r}\}$ is a spanning set for ${\bf L}/{\bf K}$. 
(Note that $\Omega$ will be a basis for ${\bf L}/{\bf K}$ if $[{\bf K(\beta_j)} : {\bf K}] = [{\bf K(\beta_1,\dots, \beta_j)} : {\bf K}(\beta_1,\dots, \beta_{j-1})]$ for all $j$; see \cite{Lang_Algebra}*{Proposition~V.1.2}.)

With this notation, we may write each $g_i$ as
$$
    g_i = \sum_{j=1}^{r}g_{i,j}\omega_j.
$$
for $ g_{i,j} \in {\bf K}$.  

Let $z=\sum_{i=0}^{\ell-1}h_i\alpha^i,$ where $h_i \in {\bf L}$.  
As before, write 
$$
    h_i=\sum_{j=1}^{r}h_{i,j}\omega_j,
$$ where $h_{i,j} \in {\bf K}$.

Define $\hat E=E(d_{0,1},\ldots, d_{r_s-1,s}, g_{0,1},\ldots, g_{\ell-1,r}, h_{0,0},\ldots,h_{\ell-1,r})$ and $L=\hat E(\beta_1,\dots, \beta_s)$.
We claim that $\hat E$ and $L$ thus defined satisfy the requirements of the lemma.
\begin{enumerate}
\item $\hat E \subset {\bf K}$ since all elements adjoined to $E$ are in ${\bf K}$.
\item Since $L=\hat E(\beta_1,\dots,\beta_s) $ and each $\beta_j \in {\bf L}$, we have that $L \subset {\bf K}(\beta_1,\dots, \beta_s)={\bf L}$.
\item Observe that the minimal polynomial $d_{0,j}+d_{1,j}T +\ldots + T^{r_j}$ of $\beta_j$ over ${\bf K}$ has coefficients in ${\hat E}$, so it must also be the minimal polynomial of $\beta$ over $\hat E$. 
Therefore $[{\hat E}(\beta_j):\hat E]=[{\bf K}(\beta_j):{\bf K}]$ for all $j$.
\item Similarly to the previous part, $g_0+g_1T +\ldots +T^{\ell}$ is the minimal polynomial of $\alpha$ over $L$ because the coefficients are in $L$.
\item Recall $z =\sum_{i=0}^{\ell-1}h_i\alpha^i$ where 
$h_i\in L$ for all $i$, so $z\in L(\alpha)$ and we observe
\[
{\mathbf N}_{L(\alpha)/L}(z)=\prod_{j=1}^{[L(\alpha):L]}\left (\sum_{i=0}^{[L(\alpha):L]-1}h_i\alpha_j^i\right )={\mathbf N}_{{\bf L(\alpha)}/{\bf L}}(z),
\]
where $\alpha_1:=\alpha, \ldots,\alpha_{[L(\alpha):L]}$ are the conjugates of $\alpha$ over $L$. The final equality comes from the fact that these are also the conjugates of $\alpha$ over ${\bf L}$.
\end{enumerate}
\end{proof}

\section{Integrality at primes in global function fields}
\label{sec:global}
Throughout this section, let $K$ be a global function field.
Recall that the Hasse Norm Principle, which we will use below, only applies to cyclic extensions.
Therefore, we assume throughout the section that $K$ contains a primitive $q$-th root of unity, which implies that every extension generated by a $q$-th root is indeed cyclic.
However, we remind the reader that the definability results remain true even if $K$ does not contain this root of unity; see Corollary~\ref{cor:zeta_in_K}. 

\subsection{Radical extensions and the behavior of primes.}
Recall that, given the base global field field $K$, we are pursuing definability via norm equations over a larger field $L$ which is generated over $K$ by some $q$-th roots, as in \eqref{eq:gen_norm_eq}.
We use the freedom of choosing $L$ to ensure that both  Proposition~\ref{prop:hasse_application} and Lemma~\ref{le:not_norm} are applicable, meaning that we have control over which norm equations are solvable and which are not.
Our method exploits Lemma~\ref{le:split1} to control the behavior of primes in $q$-th root extensions, as needed.
The next three lemmas are building blocks which will be combined when defining $L$ in the norm equations that follow.

\begin{lemma}
\label{le:ramify}
    Fix an element $z \in K^\times$ and define $L=K(\sqrt[q]{1+z^{-1}})$. The following are true.
    \begin{enumerate}[(a)]
        \item  If $\pp_K$ is a prime of $K$ such that $\ord_{\pp_K}z<0$ and $\ord_{\pp_K}z \not \equiv 0 \bmod q$, then the prime $\pp_K$ splits completely in $L$. In particular, if $\pp_L$ is a factor of $\pp_K$ in $L$, then $\ord_{\pp_L}z<0$ and $\ord_{\pp_L}z \not \equiv 0 \bmod q$.
        \item For any prime $\ttt_L$ of $L$ such that $\ord_{\ttt_L}z\geq0$ we have that $\ord_{\ttt_L}z\equiv 0 \bmod q$. 
    \end{enumerate} 
\end{lemma}
\begin{proof}
We first consider the case where $L=K$.  The first assertion of the lemma is clearly true in this case.  To see that the second part is also true,  note that $L=K$ if and only if there exists $u \in K$ such that $u^q=(1+z^{-1})$.  In this case, $1+z^{-1}$ has order 0 mod $q$ at all primes.  In particular, if some prime $\ttt_K$ is a zero of $z$, then $\ord_{\ttt}z=-\ord_{\ttt}(1+z^{-1})\equiv 0 \mod q$.  

We now assume that $K\ne L$.  Since $1+ z^{-1} \equiv 1 \mod \pp_K$, by Lemma \ref{le:split1}.\eqref{it:2} we have that $\pp_K$ splits completely in the extension $L/K$.  Therefore, for every factor $\pp_L$ of $\pp_K$ in $L$ we have that $e(\pp_L/\pp_K)=f(\pp_L/\pp_K)=1$. We conclude that $\ord_{\pp_L}z=\ord_{\pp_K}z\not \equiv 0 \bmod q$.

Let $\ttt_K$ be a $K$-prime such that $\ord_{\ttt_K}z>0$ and assume $\ord_{\ttt_K}z \not \equiv 0 \bmod q$.  
Then $\ord_{\ttt_K}(1+z^{-1})<0$ and $\ord_{\ttt_K}(1+z^{-1}) \not \equiv 0 \bmod q$.  Therefore, $\ttt_K$ ramifies totally in the extension $L/K$  by Lemma~\ref{le:split1}.\eqref{it:3}. Thus, $\ord_{\ttt_L}z=q\ord_{\ttt_K}z \equiv 0 \bmod q$ for any $L$-prime $\ttt_L$ lying over $\ttt_K$.
\end{proof}

\begin{lemma}
\label{le:split2}
   Fix $a, w \in K^\times$ and assume that for any prime $\aaa_K$ occurring in the divisor of $a$, we have $\ord_{\aaa_K}w \geq 0$ and $\ord_{\aaa_K}w \equiv 0\bmod q$.  
   If $L=K(\sqrt[q]{1+(a+a^{-1})w^{-1}})$, then the following are true.
   \begin{enumerate}[(a)]   
       \item If $\pp_K$ is a prime of $K$ such that $\ord_{\pp_K}w<0$ and $\ord_{\pp_K}w \not \equiv 0 \bmod q$, then $\pp_K$ splits completely in $L$.
       In particular, for any $L$-factor $\pp_L$ of $\pp_K$ we have that $\ord_{\pp_L}w <0$ and $\ord_{\pp_L}w \not \equiv 0 \bmod q$.\label{it:split2:a}
        \item For any prime $\ttt_L$ of $L$ we have that $\ord_{\ttt_L}a \equiv 0\bmod q$. \label{it:split2:b}
   \end{enumerate}
   \end{lemma}
   
   \begin{proof}
   For the first part, we assume that $K \ne L$ because the claim is trivial otherwise. 
Observe that $1+(a+a^{-1})w^{-1}\equiv 1 \bmod \pp_K$.  
Therefore in the extension $L/K$ the prime $\pp_K$ splits completely by Lemma~\ref{le:split1}.\eqref{it:2}.  
Hence $\ord_{\pp_L}w=\ord_{\pp_K}w$, 
proving \eqref{it:split2:a}.

Let $\ttt_K$ be the prime of $K$ below $\ttt_L$.
It suffices to consider the case when $\ttt_K$ is a zero or pole of $a$ in $K$, in which case it is a pole of $a+a^{-1}$. We observe $\ord_{\ttt_K}(a+a^{-1})=-|\ord_{\ttt_K}a|$.
By assumption, $\ord_{\ttt_K}w \geq 0$, so $\ttt_K$ is also a pole of $(a+a^{-1})w^{-1}$
Moreover, because $\ord_{\ttt_K}w\equiv 0\bmod q$, we obtain
\[
-|\ord_{\ttt_K}a| \equiv \ord_{\ttt_K}(a+a^{-1})w^{-1} = \ord_{\ttt_K}(1+(a+a^{-1})w^{-1})\bmod q
\]
If $L=K$, then $1+(a+a^{-1})w^{-1}$ is a $q$-th power, and therefore has order divisible by $q$ at all primes. 
Otherwise, if $L\neq K$, then we have $\ord_{\ttt_L}a = e(\ttt_L/\ttt_K)\ord_{\ttt_K}a $ and we have $e(\ttt_L/\ttt_K) =q$ whenever $-|\ord_{\ttt_K}a| \equiv \ord_{\ttt_K}(1+(a+a^{-1})w^{-1}) \not \equiv 0 \bmod q$ by Lemma~\ref{le:split1}.\eqref{it:3}. 
In either case, we have that $\ord_{\ttt_L}a \equiv 0 \bmod q$.
   \end{proof}

\begin{lemma}
\label{le:pole}
Fix $x,b \in K^\times$, and let $\pp_K$ be a prime of $K$.  Assume $\ord_{\pp_K}b<0$ and $\ord_{\pp_K}b \not \equiv 0 \bmod q$. The following are true.
\begin{enumerate}[(a)]
    \item Unconditionally, $\ord_{\pp_K}(bx^q+b^q)<0.$
    \item If either $\ord_{\pp_K} x \geq 0$ or $q\ord_{\pp_K} x+\ord_{\pp_K}b < q\ord_{\pp_K}b$, then  \[
 (\ord_{\pp_K}x<0) 
 \iff (\ord_{\pp_K}(bx^q+b^q)\not \equiv 0 \bmod q).
 \]
\end{enumerate}
\end{lemma}
\begin{proof}
Observe that, given our assumptions on the orders of $x$ and $b$, we have that $\ord_{\pp_K}bx^q \ne \ord_{\pp_K}b^q$.  
Thus, $\ord_{\pp_K}(bx^q+b^q)=\min(\ord_{\pp_K}bx^q,\ord_{\pp_K}b^q)<0$.

If $\ord_{\pp_K}x <0$, then, given our assumptions on the orders of $x$ and $b$  at $\pp_K$, we have that 
$$
    \ord_{\pp_K}(bx^q+b^q)=\ord_{\pp_K}bx^q = q\ord_{\pp_K} x +\ord_{\pp_K}b \not \equiv 0 \bmod q.
$$
Conversely, if $\ord_{\pp_K}x \geq 0$ then $\ord_{\pp_K}(bx^q+b^q)=\ord_{\pp_K}b^q  \equiv 0 \bmod q$.
\end{proof}

\subsection{Defining valuation rings in global function fields}
As we have mentioned in the history section of the introduction, the existential definability of valuation rings was first established by Rumely in \cite{Rumely80}.  In this section we produce a different definition using a ``floating'' base for the norm equations.  We then adapt our definition to apply in infinite extensions and for the purposes of defining integral closures of rings of $\calS$-integers. The goal of this subsection is to construct this alternative form of a definition of valuation rings via norm equations.

To begin, we define the set of elements that have a large order relative to some element $b$. 
This leads to the definability of the valuation ring by choosing $b$ to have order precisely $-1$; see Corollary~\ref{cor:val_ring_global}.
\begin{proposition}
\label{prop:L2/L1}
Fix $a,b \in K^\times$. Let $\pp_K$ be a $K$-prime, and assume the following.
\be
    \item $\ord_{\pp_K}b<0$ and $\ord_{\pp_K}b \not \equiv 0 \bmod q$;
    \item $a$ is a unit at $\pp_K$ and $a$ is not a $q$-th power modulo $\pp_K$;
    \item $a \equiv 1$ modulo any prime $\bb_K \ne \pp_K$ such that $\ord_{\bb_K}b\ne 0$.
\ee
Given $x\in K^\times$, define $L_1=K(\sqrt[q]{1+(bx^q+b^q)^{-1}})$ and
$L_2=L_1(\sqrt[q]{(1+(a+a^{-1})b^{-1})}$.
Under the assumptions on $a$ and $b$, the norm equation
$$
    {\bf N}_{{L_2(\sqrt[q]{a})}/L_2}(y)=bx^q+b^q
$$
 has a solution $y\in L_2(\sqrt[q]{a})$ if and only if $\ord_{\pp_K} x > \frac{q-1}{q}\ord_{\pp_K}b$.
\end{proposition}
\begin{proof}
    First, note that it is impossible to have $\ord_{\pp_K} x = \frac{q-1}{q}\ord_{\pp_K}b$ because we have $\ord_{\pp_K}b\not\equiv 0\bmod q$ by hypothesis. Thus, we consider only the strict inequalities.
    
    Suppose that $\ord_{\pp_K}x <\frac{q-1}{q}\ord_{\pp_K}b <0$.  
    By Lemma \ref{le:pole}, we see  $\ord_{\pp_K}(bx^q+b^q)<0$ and $\ord_{\pp_K}(bx^q+b^q)\not\equiv 0 \bmod q$.  
    Therefore, $\pp_K$ splits completely in $L_2$ by applying Lemma~\ref{le:ramify} and \ref{le:split2} in succession.
    Hence if $\pp_{L_2}$ is an $L_2$-factor of $\pp_K$, then $\ord_{\pp_{L_2}}(bx^q+b^q)\not\equiv 0 \bmod q$ and $a$ is not a $q$-th power in the residue field of $\pp_{L_2}$.
    This implies that $\pp_{L_2}$ is inert in the extension ${L_2(\sqrt[q]{a})}/L_2$ by Lemma \ref{le:split1}.\eqref{it:2}, and $bx^q+b^q$ is not an ${L_2(\sqrt[q]{a})}$-norm in $L_2$ by Lemma~\ref{le:not_norm}.

    Now suppose now that $\ord_{\pp_K}x > \frac{q-1}{q}\ord_{\pp_K}b$.
    Observe that the $L_2$-divisor of $a$ is the $q$-th power of another divisor by Lemma \ref{le:split2}.
    Thus, we may use the application of the Hasse Norm Principle given in Proposition~\ref{prop:hasse_application} to show that the norm equation has a solution. 
    In particular, it suffices to show that $bx^q + b^q$ has order divisible by $q$ at every prime which does not split completely in ${L_2(\sqrt[q]{a})}/L_2$. 
    
    Clearly, we only need to check the order of $bx^q + b^q$ primes which are zeros and poles of $bx^q + b^q$, since the order of $bx^q + b^q$ is 0 at all other primes. To start,  the order of every zero of $bx^q+b^q$ in $L_1$ is divisible by $q$ by Lemma \ref{le:ramify}.  For the distinguished pole $\pp_K$, the hypothesis $\ord_{\pp_K}x > \frac{q-1}{q}\ord_{\pp_K}b$ implies 
     $$
     \ord_{\pp_K}(bx^q+b^q) = \min\{\ord_{\pp_K}(bx^q),\ord_{\pp_K}(b^q)\}=\ord_{\pp_K}(b^q) \equiv 0 \bmod q.
     $$

    To finish the proof, consider a pole $\cc_K$ of $bx^q + b^q$ in $K$ which is different from $\pp_K$.
    If $\cc_K$ is a zero or pole of $b$, then every factor $\cc_{L_2}$ of $\cc_K$ in $L_2$ splits completely in $L_2(\sqrt[q]{a})$ by Lemma~\ref{le:split1}.\eqref{it:2} because $a\equiv 1\bmod \cc_K$ by hypothesis.
    Otherwise, $\ord_{\cc_K}b = 0$ and $\ord_{\cc_K}x < 0$. 
    We conclude that
    $$
        \ord_{\aaa_K}(bx^q + b^q) = \ord_{\aaa_K}x^q \equiv 0 \bmod q.
    $$
This finishes the casework, so the norm equation is solvable by Proposition~\ref{prop:hasse_application}.
\end{proof}

Although we are ultimately concerned with defining rings in extensions of $\F_p(t)$ which are possibly infinite, we pause to deduce a simple corollary regarding valuation rings of the global field $K$ itself.
The following lemma provides the existence of the necessary elements $a$ and $b$ for applying the preceding proposition in the case of global fields.

\begin{lemma}
\label{le:exists}
    For any $\pp_K$ there exists a pair $(a,b) \in K^2$ such that $\ord_{\pp_K}b=-1$, $a \equiv 1$ modulo any prime occurring in the divisor of $b$ not equal to $\pp_K$ and $a$ is not a $q$-th power modulo $\pp_K$.
\end{lemma}
\begin{proof}

Let $\pi_{\pp_K}$ be the local uniformizing parameter for $\pp_K$.  By the Weak Approximation Theorem, there is $b\in K$  so that $\ord_{\pp_K}(b^{-1}-\pi_{\pp_K}) > 1$, hence $\ord_{\pp_K}b=-1$; see the Weak Approximation Theorem \cite[Theorem I.1.1]{Lang94}. 

Let $\qq_1, \ldots, \qq_r$ be all the primes occurring in the divisor of $b$ and not equal to $\pp_K$. Let $w \in K$ be such that $\ord_{\pp_K}w=0$ and the equivalence class of $w$ in the residue field of $\pp_K$ is not a $q$-th power.  By applying the Weak Approximation Theorem again, there exists an element $a \in K$ such that $\ord_{\pp_K}(a-w) >0$ and  $\ord_{\qq_i}(a-1)>0, i=1,\ldots, r$.  Then $a$ satisfies the requirements stated in the lemma.
\end{proof}

By putting together the preceding proposition and lemma, we deduce the following.
\begin{corollary}
\label{cor:val_ring_global}
    There exists $a,b\in K^\times$ so that the valuation ring $\OO_{\pp_K}$ of $K$ at $\pp_K$ is defined in $K$ by the following diophantine formula:
    \begin{equation}
        \exists y \in L_2(\sqrt[q]{a}) \ {\bf N}_{{L_2(\sqrt[q]{a})}/L_2}(y)=bx^q+b^q.
        \label{eq:cor:val_ring_global}
    \end{equation}
\end{corollary}
\begin{proof}
Recall that formula \eqref{eq:cor:val_ring_global} represents a diophantine formula over $K$ by Lemma~\ref{le:norm_as_diophantine}.
By Lemma~\ref{le:exists}, there is a pair $(a,b)\in K^2$ satisfying the requirements of Proposition~\ref{prop:L2/L1} such that $\ord_{\pp_K}b = -1$.
Thus, the formula 
\eqref{eq:cor:val_ring_global}
    defines the set of all $x\in K$ which satisfy $\ord_{\pp_K}x \geq \lceil-\frac{q-1}{q}\rceil = 0$ by Proposition~\ref{prop:L2/L1}.
\end{proof}

\subsection{Defining rings of integral functions in global function fields}

To finish this section, we provide a definition of the ring of 
$\calS$-integral functions in global function fields in Proposition~\ref{prop:finite}.
The two lemmas below will also be used  to prove the definability of rings of $\calS$-integral functions in $q$-bounded fields; see Proposition~\ref{thm:infinite}.

\begin{lemma}
\label{le:nosolution}
 Let $\pp_K$ be a $K$-prime and assume the elements $a,b,x\in K^\times$ satisfy the following properties.
 \begin{enumerate}
     \item $\ord_{\pp_K}b \not \equiv 0 \bmod q$ and $\ord_{\pp_K}b <0$;
     \item $a$ is a unit at $\pp_K$ and is not a $q$-th power modulo $\pp_K$
 \end{enumerate}
 Define $L_1=K(\sqrt[q]{1+(bx^q+b^q)^{-1}})$, $L_3=L_1(\sqrt[q]{1+x^{-1}})$, and $L_4=L_3(\sqrt[q]{(1+(a+a^{-1})x^{-1})}$.
If $\ord_{\pp_K}x<\frac{q-1}{q}\ord_{\pp_K}b$, then
 the norm equation below is not solvable over $L_4(\sqrt[q]{a})$:
    \[
    {\mathbf N}_{L_4(\sqrt[q]{a})/L_4}(y)=bx^q+b^q.
    \]    
\end{lemma}
\begin{proof}
For this proof, we will apply Lemma~\ref{le:split1} repeatedly as we go up the tower of fields. Our goal is to show that the relative degree of $\pp_{L_4}$ over $\pp_K$ is one. Observe that in this case $a$ is not a $q$-th power in the residue field of $\pp_{L_4}$ and therefore not a $q$-th power in $L_4$.

First, by assumption $\ord_{\pp_K}(bx^q+b^q)=\ord_{\pp_K} bx^q <0$.  Thus, 
$$
    1+(bx^q+b^q)^{-1} \equiv 1 \bmod \pp_K,
$$
so $\pp_K$ splits completely in $L_1$ by Lemma~\ref{le:split1}.\eqref{it:2} regardless of whether $L_1=K$ or not.
Similarly, if $\pp_{L_1}$ is an $L_1$-factor of $\pp_K$, then $1+x^{-1} \equiv 1 \bmod \pp_{L_1}$  and therefore $\pp_{L_1}$ splits completely in $L_3$, again by Lemma~\ref{le:split1}.\eqref{it:2} whether $L_3=L_1$ or not.  
Since $a$ is a unit at $\pp_K$, we have that $(a+a^{-1})$ does not have a pole at any $L_3$-factor $\pp_{L_3}$ of $\pp_K$.  
Therefore, for any such $\pp_{L_3}$, we have that $\ord_{\pp_{L_3}}(a+a^{-1})x^{-1} >0$ and $\pp_{L_3}$ splits completely in $L_4$, again by Lemma~\ref{le:split1}.\eqref{it:2} whether $L_4=L_3$ or not.  

Therefore for any $L_4$-prime $\pp_{L_4}$ lying above $\pp_K$, we have that $f(\pp_{L_4}/\pp_K)=1$ and as pointed out above $a$ is not a $q$-th power in the residue field of $\pp_{L_4}$ and not a $q$-th power in $L_4$.  
Therefore, $\pp_{L_4}$ is inert in $L_4(\sqrt[q]{a})$ by Lemma~\ref{le:split1}.\eqref{it:2}.
Because $\ord_{\pp_{L_4}}(bx^q+b^q) \not \equiv 0 \bmod q$, the norm equation 
does not have a solution in $L_4(\sqrt[q]{a})$ by Lemma~\ref{le:not_norm}.
\end{proof}

\begin{lemma}
\label{le:existsolutions}
  Fix $x \in K^\times$ and let $\calS_K$ be a nonempty finite set of valuations of $K$ containing all valuations occurring as poles of $x$. Let $a, b \in K^\times$ satisfy the following conditions:
  \begin{enumerate}
      \item $a$ and $b$ are units at all valuations of $\calS_K$,
      \item $a \equiv 1$ modulo all valuations in $\calS_K$.
  \end{enumerate}
  Let $L_1=K(\sqrt[q]{1+(bx^q+b^q)^{-1}})$, 
  $L_3=L_1(\sqrt[q]{1+x^{-1}})$, 
  and $L_4=L_3(\sqrt[q]{1+(a+a^{-1})x^{-1}})$.
  Then the norm equation below has solutions in $L_4(\sqrt[q]{a})$:
   \[
    {\mathbf N}_{L_4(\sqrt[q]{a})/L_4}(y)=bx^q+b^q.
    \]    
\end{lemma}
\begin{proof}
    We will use the application of the Hasse Norm Principle found in Proposition~\ref{prop:hasse_application}.
    First, let $\aaa_{L_1}$ be a prime in the $L_1$-divisor of $a$. 
    By hypothesis, $a$ is a unit at all poles of $x$, so $\aaa_{L_1}$ is not a pole of $x$.  
    This implies that $\ord_{\aaa_{L_3}}x\equiv 0\bmod q$ by Lemma~\ref{le:ramify} applied to the extension $L_3/L_1$ for any prime $\aaa_{L_3}$ lying over $\aaa_{L_1}$. 
    Therefore, $\ord_{\aaa_{L_4}}a\equiv 0\bmod q$ for any prime $\aaa_{L_4}$ of $L_4$ lying over $\aaa_{L_3}$ by applying Lemma~\ref{le:split2} to the extension $L_4/L_3$,  so this proves that the $L_4$-divisor of $a$ is the $q$-th power of another divisor.

    To finish the proof, we show that every prime $\pp_{L_4}$ of $L_4$ is either split in $L_4(\sqrt[q]{a})$ or satisfies $\ord_{\pp_{L_4}}(bx^q + b^q)\equiv 0\bmod q$. Without loss of generality, we only need to consider primes in the divisor of $bx^q+b^q$.
    If $\pp_K$ is a zero of $bx^q+b^q$, then Lemma~\ref{le:ramify} applied to the extension $L_1/K$ implies that $\ord_{\pp_{L_1}}(bx^q+b^q)\equiv 0\bmod q$ for every prime $\pp_{L_1}$ lying over $\pp_{K}$. 
     
     Finally, suppose $\pp_K$ is a pole of $bx^q + b^q$ in $K$. If $\pp_K\in \calS_K$, then $a\equiv 1\bmod \pp_K$ by hypothesis, so all factors  of $\pp_K$ in $L_4$ split completely in $L_4(\sqrt[q]{a})$ by Lemma~\ref{le:split1}.\eqref{it:2}. (If the extension $L_4(\sqrt[q]{a})/L_4$ is trivial, then this statement is trivially true.)
     Alternatively, if $\pp_K\not\in \calS_K$, then $\ord_{\pp_K}x\geq 0$ by hypothesis, so $\pp_K$ must be a pole of $b$ and $\ord_{\pp_K}(bx^q + b^q) = q\ord_{\pp_K}b \equiv 0\bmod q$.
     This establishes that all hypotheses of Proposition~\ref{prop:hasse_application} are satisfied, which completes the proof.
 \end{proof}

As anticipated above, the preceding lemmas, along with Weak Approximation, lead to the definability of rings of $\calS$-integers in global function fields.
Recall the notation that
 $U_{\calS_K}$ is the set of elements of ${K}$ that are units at every prime of $\calS_K$ and  $V_{\calS_K}$ is the set of elements $x$ in  ${K}$ such that $\ord_{\pp_K}(x-1)>0$ for every $\pp_K\in\calS_K$.

\begin{proposition}
\label{prop:finite}
   Let $\calS_K$ be a nonempty finite set of valuations of $K$. For each choice of $a,b,x\in K^\times$, define the fields 
   $L_1=K(\sqrt[q]{1+(bx^q+b^q)^{-1}})$, 
   $L_3=L_1(\sqrt[q]{1+x^{-1}})$, 
   and $L_4=L_3(\sqrt[q]{(1+(a+a^{-1})x^{-1})}$.
   The sentence 
    \[
    \forall a \in \calV_{\calS_K} \forall b\in U_{\calS_K}\exists y \in L_4(\sqrt[q]{a}): {\mathbf N}_{L_4(\sqrt[q]{a})/L_4}(y)=bx^q+b^q
    \]
    is true if and only if $x$ is an $\calS_K$-integer.
\end{proposition}
\begin{proof}
  Let $\pp_K \not \in \calS_K$  be such that $\ord_{\pp_K}x<0$.  We show that there exists $a \in \calV_{K}\cap U_{\{\pp_K\}}$, and $b \in U_{\calS_K}$ such that norm equation does not have a solution $y$ in $L_4(\sqrt[q]{a})$. By the Weak Approximation Theorem we can choose $b \in U_{\calS_K}$ such that $\ord_{\pp_K}b \not \equiv 0\bmod q$, $\ord_{\pp_K}b<0$ and $\ord_{\pp_K}bx^q<\ord_{\pp_K}b^q$; see \cite[Theorem I.1.1]{Lang94}.
  Similarly, by the same theorem, there exists $a \in \calV_{K}$ such that $a$ is  a unit modulo $\pp_K$ and not a $q$-th power modulo $\pp_K$. 
  By Lemma \ref{le:nosolution}, it now follows that the norm equation does not have solution for this choice of $a$ and $b$.

  Assume now that all poles of $x$ are in $\calS_K$, and $a, b$ are units at valuations in $\calS_K$ such that $a$ is equivalent to 1 modulo valuations in $\calS_K$.  It now follows that the norm equation has solutions by Lemma \ref{le:existsolutions}.  
\end{proof}

\section{\texorpdfstring{$q$-Bounded fields}{q-Bounded fields}}
\label{sec:q_bounded_def}

Now that we have established some definability and decidability results for finite extensions of $\F_p(t)$ in Section~\ref{sec:global}, we turn our attention to the class of infinite algebraic extensions of $\F_p(t)$ which we call $q$-bounded fields. 
In the introduction, we presented the concept of \emph{global $q$-boundedness}  (Definition~\ref{def:global_q_bounded}), whose name reflects the fact that the property concerns degrees of global extensions $\hat K\supseteq K \supseteq \F_p(t)$.
In reality, the properties that we exploit are local conditions concerning primes $\pp_K$ and their corresponding completions $K_{\pp_K}$, which we call local fields.
The present section is dedicated to developing this idea.
We ultimately present three equivalent ways to conceptualize $q$-boundedness; see Proposition~\ref{prop:equiv_q_bounded_defs}.

\subsection{\texorpdfstring{$q$}{q}-Bounded algebraic extensions of local fields}

We first provide a local definition, which will pave the way for subsequently defining $q$-boundedness for algebraic extensions of $\F_p(t)$.
We use the term \emph{local function field} to refer to a completion $K_\pp$, where $K$ is a finite extension of $\F_p(t)$ and $\pp$ is a prime of $K$.
In the local setting, $q$-boundedness is merely a restriction on the degrees of finite subextensions.

\begin{definition}[\texorpdfstring{$q$}{q}-Bounded algebraic extensions of local function fields]
\label{def:q_bounded_local}
 If ${\bf L}_{\pp}$ is an algebraic extension of a local function field $K_{\pp}$, then we say that ${\mathbf L}_{\pp}$ is $q$-bounded if there exists some $r \in \Z_{>0}$ so that $\ord_q([H_{\pp}:K_{\pp}]) \leq r$ for every finite extension $H_{\pp}$ of $K_{\pp}$ contained in ${\bf L}_\pp$.
\end{definition}

Equivalently, $q$-boundedness in the local setting can be formulated in terms of the existence of a certain maximal finite extension.
\begin{lemma}
\label{le:equivalent}
  If ${\bf L}_{\pp}$ is an algebraic  extension of a local function field $K_{\pp}$, then the following are equivalent
  \begin{enumerate}
      \item ${\bf L}_{\pp}$ is $q$-bounded;\
      \item There exists a finite extension $H_\pp$ of $K_\pp$ in ${\bf L}_{\pp}$ such that $\ord_q([E_{\pp}:H_\pp])=0$ for every finite extension $E_{\pp}/H_{\pp}$ contained in ${\bf L}_{\pp}$. 
  \end{enumerate}
\end{lemma}
\begin{proof}
First assume that ${\bf L}_{\pp}$ is $q$-bounded.  Let 
\[
r=\max\{\ord_q[F_{\pp}:K_{\pp}] : F_{\pp}/K_{\pp} \text{ is finite }, F_{\pp} \subset {\bf L}_{\pp}\}.
\]    
By assumption $r<\infty$ and therefore there is a field $H$ such that $\ord_q[H_{\pp}:K_{\pp}]=r$.  By the definition of $r$ and the multiplicativity of degrees in towers, it now follows that $\ord_q[E_{\pp}:H_{\pp}]=0$ for every finite extension $E_{\pp}/H_{\pp}$ contained in  ${\bf L}_{\pp}$.

Now assume the field $H_{\pp}$ exists.  If $E_{\pp} \subset {\bf L}_{\pp}$ is a finite extension of $K_{\pp}$, then 
\[
\ord_q[F_{\pp}:K_{\pp}]\leq \ord_q[F_{\pp}H_{\pp}:K_{\pp}]=\ord_q[H_{\pp}:K_{\pp}]+\ord_q[F_{\pp}H_{\pp}:H_{\pp}]=\ord_q[H_{\pp}:K_{\pp}]
\]
by the definition of $H_{\pp}$, which completes the proof.
\end{proof}

\subsection{\texorpdfstring{$q$}{q}-Boundedness for algebraic extensions of global function fields}

Our next goal is to define $q$-boundedness for global function fields in terms of the local property outlined in the previous subsection.
Rather than appealing to the completions of arbitrary algebraic extensions of $\F_p(t)$, it is sufficient to work with completions of global fields and their unions in towers, as described in the following definition.

\begin{definition}[Tower of completions]
\label{def:tower_completion}
 Let ${\bf K}$   be an algebraic extension of $\F_p(t)$. 
 Write ${\bf K} = \cup_{j = 0}^\infty K_j$ where 
 $$
    \F_p(t) \subseteq K_1\subseteq K_2\subseteq \dots
$$
is a tower of finite extensions.
 Let $v$ be a valuation on $\F_p(t)$, let ${\bf v}$ be an extension of $v$ to ${\bf K}$, and let $v_i$ be the restriction of ${\bf v}$ to $K_i$.
 Let $\iota:{\bf K}\hookrightarrow {\bf K}_{{\bf v}}$ be an embedding and let $(\iota(K_i))_{v_i}=(\iota(K_i))_{{\bf v}}$ be the completion of $\iota(K_i)$ inside ${\bf K}_{{\bf v}}$ with respect to $v_i$ for each $i\geq 1$.
 We say that $\bigcup_{i=0}^{\infty}(\iota(K_i))_{v_i}$ is a tower of completions of ${\bf K}$ with respect ${\bf v}$, and we denote it by ${\bf K}_{T,{\bf v}}$.
\end{definition}

Intuitively, the tower of completions ${\bf K}_{T,{\bf v}}$ is the smallest extension of $K_{v_K}$ encapsulating all the local requirements for $q$-boundedness.
See the appendix (Section \ref{ssec:valuations}) for a discussion of the difference between the tower of completions ${\bf K}_{T,{\bf v}}$ and the completion ${\bf K_v}$ of ${\bf K}$.
As indicated by our choice of notation, the field ${\bf K}_{T,{\bf v}}$ only depends on the field ${\bf K}$ and the specified valuation ${\bf v}$, as shown in the following lemma.

\begin{lemma}
    In the notation of Definition~\ref{def:tower_completion}, the tower of completions ${\bf K}_{T,{\bf v}}$ is independent of the choice of the tower of global fields $K_1\subseteq K_2\subseteq \dots $ and the choice of transcendental element $t$.
\end{lemma}
\begin{proof}
    Consider another tower of global function fields
   $L_1\subseteq L_2\subseteq \dots $ satisfying ${\bf K}=\cup_{j = 0}^\infty L_j$, and let $w_i$ be the restriction of ${\bf v}$ to $L_i$.  
   It is enough to note that for all $i$ there exists $j$ so that $L_i \subset K_j$ because each $L_i$ is finitely generated over $\F_p(t)$. 
    Consequently, $\iota(L_i) \subset \iota(K_j)$ and $(\iota(L_i))_{w_i}\subset (\iota(K_j))_{v_j}$ since both valuations are restrictions of ${\bf v}$.
    We conclude that $$\bigcup_{i=0}^{\infty}(\iota(K_i))_{v_i}\supseteq \bigcup_{i=0}^{\infty}(\iota(L_i))_{w_i}.$$
    The reverse containment is clear from reversing the roles of $\{K_i\}$ and $\{L_j\}$.

    Now let $t_1, t_2 \in {\bf K}$ be two transcendental elements over $\F_p$. 
    Observe that $\F_p(t_1,t_2)$ is a finite algebraic extension of $\F_p(t_j)$ inside ${\bf K}$ for each $1\leq j\leq 2$.  
    Therefore, there is a tower extending from $\F_p(t_j)$ to ${\bf K}$ containing $\F_p(t_1,t_2)$ for both $j=1$ and $j=2$, and these towers give rise to the same field ${\bf K}_{T,{\bf v}}$ by construction.
\end{proof}

\begin{proposition}
\label{prop:existalpha}
Let $K$ be a global function field, let ${\bf K}$ be an algebraic extension of $K$, and let $\pp_{\bf K}$ be a prime of ${\bf K}$ lying over a prime $\pp_K$ of $K$. If $\tt F$ is a finite extension of $K_{\pp_K}$ inside ${\bf K}_{T,\pp_{\bf K}}$, then there is a finite extension $F_0$ of $K$ contained in ${\bf K}$ so that  $(F_0)_{\pp_{F_0}} \supseteq {\tt F}$, where $\pp_{F_0}$ is the prime of $F_0$ lying under $\pp_{\bf K}$.
\end{proposition}
\begin{center}
\begin{tabular}{c}
\xymatrix{ {\bf K} \ar[r]^{\iota} & {\bf K}_{T,\pp_{\bf K}} \\
            F_0\ar[r]\ar[u]^{\text{infinite}}& (F_0)_{\pp_{F_0}}\ar[u]_{\text{infinite}}\\
            & {\tt F} \ar[u]_{\text{finite}}\\
            K \ar[uu]^{\text{finite}} \ar[r] & K_{\pp_{K}} \ar[u]_{\text{finite}}
              }\\
             {\tiny {\tt F} is a finite extension of $K_{\pp_{K}}$}\\
              {\tiny $\pp_{\bf K}$ is a valuation of ${\bf K}$}\\
              {\tiny $\pp_{K}, \pp_{F_0}$ are restrictions of $\pp_{\bf K}$ to $K$ and $F_0$ respectively.}
\end{tabular}
\end{center}
\begin{proof}
   By definition of ${\bf K}_{T,\pp_K}$, there exists a sequence of finite extensions $\{K_i\}$ of $K$ such that $K=K_0$ and ${\bf K}_{T,\pp_{\bf K} }=\bigcup_{i=0}^{\infty}K_{i,\pp_i}$, where $\pp_i =\pp_{\bf K}\cap K_i$.  Since ${\tt F}$ is finitely generated over $K_{\pp_K}$, we have that there exists a $j$ such that ${\tt F} \subset K_{j,\pp_j}$. 
\end{proof}

We have now established the technical background for providing our first presention of the concept of $q$-boundedness for algebraic extensions of $\F_p(t)$.

\begin{definition}[\texorpdfstring{$q$}{q}-boundedness]
\label{def:q-bounded_tower_of_completions}
Let ${\bf K}$  be an algebraic  extension of $\F_p(t)$, and let ${\bf v}$ be a valuation of ${\bf K}$.
\begin{enumerate}
    \item We call ${\bf v}$  (or its valuation ideal) $q$-bounded if the corresponding tower of completions is $q$-bounded in the sense of Definition~\ref{def:q_bounded_local}.
    \item If $K$ is a subfield of ${\bf K}$ and $v$ is a valuation of $K$, then we say $v$ (or its valuation ideal) is $q$-bounded in ${\bf K}$ if every valuation of ${\bf K}$ over $v$ is $q$-bounded.
    \item We say that ${\bf K}$ is $q$-bounded if every of valuation of ${\bf K}$ is $q$-bounded. 
\end{enumerate}  
\end{definition}

Although towers of completions give one way to understand $q$-boundedness as a local property, there are other equivalent definitions. 
We present two equivalent definitions in the proposition below.

\begin{proposition}
\label{prop:equiv_q_bounded_defs}
    Let {\bf K} be an algebraic extension of $\F_p(t)$.  For any valuation ${\bf v}$ of ${\bf K}$, the following are equivalent:
    \begin{enumerate}
        \item  There exists a finite extension $K$ of $\F_p(t)$ such that for any finite extension $L/K$ with $L \subset {\bf K}$ we have that $\ord_qe(v_{|L}/v_{|K})=0$ and $\ord_qf(v_{|L}/v_{|K})=0$.
        \label{prop:equiv_q_bounded_defs:path}
        \item The set $\{\ord_q[F_v : \F_p(t)_v]\}$ is bounded, where $F$ ranges over all finite extension of $\F_p(t)$ contained in ${\bf K}$.
        \label{prop:equiv_q_bounded_defs:local_deg}
        \item The valuation $v$ is $q$-bounded.
         \label{prop:equiv_q_bounded_defs:tower}
    \end{enumerate}
\end{proposition}
\begin{proof}
    Recall that the degree of a finite local extension is equal to the product of the corresponding ramification and relative degrees.
    In particular, for any finite extension $F$ of $\F_p(t)$ in ${\bf K}$, we have
    $$
        \ord_q[F_v : \F_p(t)_v] = \ord_qe(v_{|F}/v_{|\F_p(t)})+\ord_qf(v_{|F}/v_{|\F_p(t)})
    $$
    Therefore, the order at $q$ of ramification and relative degrees
    is bounded if and only if the order at $q$ of finite local degrees is bounded, proving the equivalence of the first two properties. 

    Assume that the second condition holds. Let $\F_p(t) = K_0\subset K_1 \dots$ be a sequence of fields with $\cup_{i = 0}^\infty K_i = {\bf K}$ and let ${\bf K}_{{\bf v}, T}$ be the corresponding tower of completions. Let ${\tt F}$ be a finite extension of $\F_{p}(t)$ inside ${\bf K}$.
    By Proposition~\ref{prop:existalpha}, there is a finite extension $F/\F_p(t)$ so that $F_v\supseteq {\tt F}$.  Therefore, because $\ord_q[F_v : \F_p(t)_v]$ is bounded, we know that $\ord_q[{\tt F} : \F_p(t)_v]$ is bounded as well. This proves that ${\bf K}$ is $q$-bounded.

    Finally, if $F$ is a finite extension of $\F_p(t)$, then $F_v$ is a finite extension of $\F_p(t)_v$, so $q$-boundedness clearly implies the second property by definition.
\end{proof}

\begin{corollary}
   ${\bf K}$ is $q$-bounded if and only if on any path from $\F_p(t)$ to ${\bf K}$ through the factor tree of prime ideals, the order at $q$ of the ramification and relative degree is bounded.  We will refer to this version of $q$-boundedness as the {\it path condition}.
\end{corollary}

\subsection{Extensions of prime degree \texorpdfstring{$q$}{q} of a field of characteristic not \texorpdfstring{$q$}{q}} 
Fix two distinct primes $p$ and $q$.

\begin{proposition}
\label{prop:extq}
    If ${\bf L}_\pp$ is a $q$-bounded algebraic extension of
    a local function field $K_\pp$,
    then the algebraic closure of $\F_p$ in ${\bf L}_\pp$ has an extension of degree $q$.
\end{proposition}
\begin{proof}
    By Lemma~\ref{le:equivalent},
    there exists a finite extension $H_\pp$ of $K_\pp$, such that for any field $E_\pp$ with $H_\pp \subset E_\pp \subset {\bf L}_\pp$ and $E_\pp/H_\pp$ is finite we have that $\ord_q f(E_\pp/H_\pp)=0$.  
    Writing $\kappa_{H_\pp}$ and $\kappa_{E_\pp}$ for the residue fields of $H_\pp$ and $E_\pp$, respectively, we have 
    \[
    [\kappa_{E_\pp}:\F_p]=[\kappa_{H_\pp}:\F_p][\kappa_{E_\pp}:\kappa_{H_\pp}]
    \]
    and $\ord_q[\kappa_{E_\pp}:\F_p]=\ord_q[\kappa_{H_\pp}:\F_p]$.  

    Let $a \in \kappa_{H_\pp}\setminus (\kappa_{H_\pp})^q$.  Then $a\in  \kappa_{E_\pp}\setminus (\kappa_{E_\pp})^q$ by \cite[Theorem~VI.9.1]{Lang_Algebra}.  
    Suppose now that $a$ is a $q$-th power in the residue field $\kappa_{{\bf L}_\pp}$ of ${\bf L}_\pp$.  Then there exists a field $E_\pp$ as above such that $a \in \kappa_F^q$, leading to a contradiction.  Therefore $\kappa_{{\bf L}_\pp}$ contains elements that are not $q$-th powers in the field. Since $\F_{\bf L_\pp} \subset \kappa_{{\bf L}_\pp}$, we have that $\F_{\bf L_\pp}$ has extensions of degree $q$.
\end{proof}

\begin{corollary}
\label{cor:extq}
 If ${\bf K}$  is a $q$-bounded extension of $\F_p(t)$, then the algebraic closure of $\F_p$ in ${\bf K}$ has an extension of degree $q$.
\end{corollary}
\begin{proof}
Let $\pp_K$ be any valuation of ${\bf K}$. Since the algebraic closure of $\F_p$ in ${\bf K}_{\pp_{\bf K}}$ has an extension of degree $q$ by Proposition \ref{prop:extq} and ${\bf K} \subset {\bf K}_{\pp_{\bf K}}$, the assertion of the corollary follows.
\end{proof}

\subsection{Uniformly of \texorpdfstring{$q$}{q}-bounded primes and fields}

For some of our results, we need even more control on the behavior of primes in the infinite algebraic extension ${\bf K}$ of $\F_{p}(t)$. 
This leads to our notion of \emph{uniform} bounds.
As we will see, this uniformity is automatic in the case of Galois extensions.

\begin{definition}[Uniform $q$-boundedness]
\label{def:uniform_bounded}
 Let ${\bf K}$ be an algebraic extension of a global function field $K$ and let $\pp_K$ be a prime of $K$. In each of the following, we consider supremums where the arguments range over all finite extensions $\hat K/K$ and all primes $\pp_{\hat K}$ of $\hat K$ lying over $\pp_K$.
 \begin{enumerate}[(a)]
     \item Say that  $\pp_K$ has \emph{uniformly $q$-bounded ramification} if there is some $B \in \N$ so that $\sup\{\ord_q e(\pp_{\hat K}/\pp_K)\}\leq B$.
     \item  If the stronger condition $\sup\{e(\pp_{\hat K}/\pp_K)\}\leq B$ holds for some $B \in \N$, then we say that $\pp_K$ has \emph{uniformly absolutely bounded ramification}.
     \item Analogously, say that  $\pp_K$ has \emph{uniformly $q$-bounded relative degree} if there is some $B \in \N$ so that $\sup\{\ord_q f(\pp_{\hat K}/\pp_K)\}\leq B$.
     \item If ${\bf K}/K$ has both uniformly $q$-bounded ramification and uniformly $q$-bounded relative degree at $\pp_K$, then we say ${\bf K}/K$ is uniformly $q$-bounded at $\pp_K$.
     If ${\bf K}/K$ is uniformly $q$-bounded at all primes of $K$, then we say that the extension is uniformly $q$-bounded.
 \end{enumerate}
\end{definition}

\begin{remark}
    The presentation of $q$-boundedness described in Proposition~\ref{prop:equiv_q_bounded_defs}.\eqref{prop:equiv_q_bounded_defs:path} is similar to the definitions of uniform boundedness given above, except that Proposition~\ref{prop:equiv_q_bounded_defs}.\eqref{prop:equiv_q_bounded_defs:path} restricts to only one factor path at a time.
    Indeed, ${\bf K}/K$ is a $q$-bounded extension if, for every prime $\pp_{\bf K}$ of ${\bf K}$, there is a $B\in \N$ so that 
    \begin{equation}
        \label{rem:q_bd_uniform}
    \begin{tabular}{ccc}
        $\sup\{\ord_q e(\pp_{\bf K}\cap \hat K/\pp_K)\} \leq B $
       & \text{ and   } &
        $\sup\{\ord_q f(\pp_{\bf K}\cap \hat K/\pp_K)\} 
        \leq B$
    \end{tabular}
    \end{equation}
    where $\hat K$ ranges over all finite extensions of $K$ in ${\bf K}$ and $\pp_K := \pp_{\bf K}\cap K$. 
    In contrast, \emph{uniform} {$q$-boundedness} of ramification and relative degree means that there is a single bound $B$ so that the respective part of \eqref{rem:q_bd_uniform} holds for all primes $\pp_{\bf K}$ in ${\bf K}$ over $\pp_K$ simultaneously.
\end{remark}

Although some $q$-bounded primes may fail to be uniformly $q$-bounded, the following result shows that this technical difficulty can be mostly avoided by passing to a finite extension.

\begin{proposition}
    Let $K$ be a finite extension of $\F_p(t)$, and let ${\bf K}$ be an algebraic (possibly infinite) extension of $K$.
    If $\pp_K$ is a $q$-bounded prime of $K$, then there exists a finite (possibly trivial) extension $L$ of $K$ inside ${\bf K}$ and an  $L$-prime $\pp_L$ over $\pp_K$ such that $e(\pp_M/\pp_L)f(\pp_M/\pp_L) \not\equiv 0 \bmod q$ for every finite extension $M$ of $L$ inside ${\bf K}$ and every $M$-prime $\pp_M$ over $\pp_L$. In particular, ${\bf K}/L$ is uniformly $q$-bounded at $\pp_L$.
    \end{proposition}
    \begin{proof}
    We first consider ramification.
    Assume for the sake of contradiction that for every finite extension $L/K$ and every $L$-prime $\pp_L$ over $\pp_K$ there is a finite extension $M/L$ and an $M$-prime $\pp_M$ over $\pp_L$ so that 
    \begin{equation}
        e(\pp_M/\pp_L) \equiv 0 \bmod q. \label{eq:ram_ML}
    \end{equation} We apply this assumption repeatedly as follows. Starting with $L := K$, write $M_1$ and $\pp_{M_1}$ for the field $M$ and prime $\pp_M$, respectively, that make Equation~\ref{eq:ram_ML} true. Similarly, after the construction of $M_i$, write $M_{i+1}$ and $\pp_{M_{i+1}}$ for the field and prime which make Equation~\ref{eq:ram_ML} true when applying the assumption with $L := M_i$. This explicitly constructs a tower $K \subseteq M_1 \subseteq M_2\subseteq \dots \subseteq {\bf K}$ so that $e(\pp_{M_{i+1}}/\pp_{M_i}) \equiv 0\bmod q$ for all $i$, and therefore $\ord_q e(\pp_{M_i}/\pp_K) \geq i$ for all $i$. Therefore, $\pp_K$ is not $q$-bounded by definition. 

    Since the proposition assumes that $\pp_K$ is $q$-bounded, we deduce that there is an extension $L/K$ and an $L$-prime $\pp_L$ over $\pp_K$ so that $e(\pp_M/\pp_L) \not\equiv 0 \bmod q$ for every finite extension $M$ of $L$ inside ${\bf K}$ and every $M$-prime $\pp_M$ over $\pp_L$. By repeating the same argument, we obtain the same with relative degree $f(\pp_M/\pp_L)$ replacing ramification $e(\pp_M/\pp_L)$. The proof is completed by combining these statements.
    \end{proof}

    \begin{corollary}
    \label{cor:existunifqbound}
If ${\bf K}$ be a $q$-bounded extension of $\F_p(t)$, then there are infinitely many pairs $(K, \pp_K)$, where $K/\F_p(t)$ is a finite extension and $\pp_K$ is a uniformly $q$-bounded prime in ${\bf K}$.
    \end{corollary}

In some cases, we can determine directly that a prime is uniformly $q$-bounded without passing to an extension. We have two results of this flavor.

\begin{lemma}
    Let ${\bf K}/K$ be an algebraic extension. If $\pp_K$ is a prime of $K$ which is $q$-bounded and has only finitely many factors in ${\bf K}$, then $\pp_K$ is uniformly $q$-bounded.
\end{lemma}
\begin{proof}
    Because $\pp_K$ is $q$-bounded, we know that for each prime $\pp_{\bf K}$ of ${\bf K}$ over $\pp_K$ there is a bound $B$ so that \eqref{rem:q_bd_uniform} is satisfied. Since there are only finitely many factors $\pp_{\bf K}$ of $\pp_K$ in ${\bf K}$ by assumption, it is clear that the maximum all such $B$ is a finite number, and the prime $\pp_K$ is uniformly $q$-bounded in ${\bf K}$ by definition.
\end{proof}

Finally, in some cases all primes are uniformly $q$-bounded.

\begin{proposition}
\label{prop:Galois_q-bounded}
    Let $K/\F_p(t)$ be a finite extension such that ${\bf K}/K$ is Galois. 
    If ${\bf K}$ is $q$-bounded at a prime $\pp_K$ of $K$, then it is uniformly $q$-bounded at $\pp_K$.
\end{proposition}
\begin{proof}
    Since ${\bf K}/K$ is algebraic, we may write ${\bf K} = \cup_{j = 0}^\infty K_i$ where $K = K_0 \subseteq K_1 \subseteq \dots$ is an increasing tower of finite extensions of $K$. 
    Because ${\bf K}/K$ is Galois, we observe that each field $K_j$ is separable over $K$, and the Galois closure $L_j$ of $K_j$ over $K$ is contained in ${\bf K}$.
    Therefore, ${\bf K} = \cup_j^\infty L_j$ is represented by a tower of finite Galois extensions of $K$.
    
    Fix a prime $\ttt_{\bf K}$ of ${\bf K}$ over $\pp_K$ and define $\ttt_{L_i} := \ttt_{\bf K}\cap L_i$ for all $i$.  Because $\pp_K$ is $q$-bounded, we know that there is $B$ so that 
        $\sup_{i\in \N}\{\ord_q e(\ttt_{L_i}/\pp_K)\} \leq B $
        and 
        $\sup_{i\in \N}\{\ord_q f(\ttt_{L_i}/\pp_K)\} \leq B$; see Proposition~\ref{prop:equiv_q_bounded_defs}.

    Let $\hat K$ be any finite extension of $K$ in ${\bf K}$ and let $\pp_{\hat K}$ be a prime of $\hat K$ over $\pp_K$. 
    Because ramification indices and relative degrees grow multiplicatively in towers, we may assume without loss of generality that $\hat K = L_i$ for some $i$.
    But ramification and relative degrees are equal across all primes of $L_i$ lying over $\pp_K$ because the extension $L_i/K$ is Galois. 
    Therefore, $\ord_q e(\pp_{\hat K}/\pp_K) = \ord_q e(\ttt_{L_i}/\pp_K)\leq B$ and similarly $\ord_q f(\pp_{\hat K}/\pp_K)\leq B$, as desired.
 \end{proof}

The following relates uniform $q$-boundedness to the simpler notion of global $q$-boundedness, as defined in Definition~\ref{def:global_q_bounded}.
 \begin{proposition}
    \label{prop:global_to_uniform_q}
     Let $K/\F_p(t)$ be a finite extension and ${\bf K}/K$ an algebraic Galois extension. If ${\bf K}$ is globally $q$-bounded, then it is uniformly $q$-bounded.
 \end{proposition}
 \begin{proof}
     Let $K = L_0 \subseteq L_1 \subseteq \dots$ be a tower of Galois extensions such that $\cup_{j = 0}^\infty L_j = {\bf K}$.
     Fix a prime $\pp_K$ of $K$.
     Let $\hat K$ be any finite extension of $K$ inside ${\bf K}$.
     Without loss of generality, we may assume $\hat K$ is Galois over $K$.
     Let $\pp_{\hat K}$ be any prime of $\hat K$ lying over $\pp_K$.
     Because $\hat K$ is Galois over $K$, we have that $e(\pp_{\hat K}/\pp_{ K})$ and $f(\pp_{\hat K}/\pp_{ K})$ both divide the degree $[\hat K : K]$.
     Since ${\bf K}$ is globally $q$-bounded, we have $D = \sup\{\ord_q[\hat K : K]\} < \infty$, and we conclude that ${\bf K}$ is uniformly $q$-bounded by definition. 
 \end{proof}

\section{Defining the integral closure of a valuation ring in an infinite \texorpdfstring{$q$}{q}-bounded extension}
\label{sec:def_val_ring}

Throughout this section, let ${\bf K}$ be an algebraic extension of a global function field $K$ and let $\pp_K$ be a prime of $K$. 
The goal of this section is to define the integral closure of the valuation ring $\OO_{\pp_K}$ in ${\bf K}$ by leveraging the global results of Section~\ref{sec:global}.
As always, we assume without loss of generality that $K$ contains a primitive $q$-th root of unity; see Corollary~\ref{cor:zeta_in_K}.

Recall that we provided a diophantine definition of valuations rings of global fields in Corollary~\ref{cor:val_ring_global}. 
In this section, we extend this machinery to prove analogous results for certain infinite algebraic extensions of global function fields.
Under the most restrictive hypotheses, we obtain a diophantine definition of the integral closure of a valuation ring. 
This includes the case where ${\bf K}$ is a $q$-bounded algebraic extension of a global field $K$.
Under more general hypotheses, it is a first-order definition. 
We start with a lemma that provides the existence of the necessary auxiliary elements $a$ and $b$ which will be used to define a norm equation.

\begin{lemma}
\label{le:ab}
If $\pp_K$ has uniformly $q$-bounded ramification and relative degree in ${\bf K}$, then there exists a finite extension $F/K$ with $F \subset {\bf K}$ and $a, b \in F$ such that for any finite extension $\hat F/F$ inside ${\bf K}$ we have that $a$ and $b$ satisfy the following conditions for any factor $\pp_{\hat F}$ of $\pp_K$ in $\hat F$:
\begin{enumerate}
    \item $\ord_{\pp_{\hat F}}b<0$ and $\ord_{\pp_{\hat F}}b\not\equiv 0\bmod q$;
    \item $a$ is a unit at $\pp_{\hat F}$ and is not a $q$-th power modulo $\pp_{\hat F}$;
    \item $a \equiv 1$ modulo any prime occurring in the divisor of $b$ not lying above $\pp_K$.
\end{enumerate}
Moreover, if we make the further assumption that $\pp_K$ has uniformly absolutely bounded ramification, then we also obtain the following stronger condition:
\begin{enumerate}
    \item[(1')] $\ord_{\pp_{\hat F}}b=-1$;
\end{enumerate}
\end{lemma}
\begin{proof}
    By definition of uniformly $q$-bounded ramification and relative degree there exists a finite extension $F/K$ with $F \subset {\bf K}$ such that the following holds for any factor $\pp_F$ of $\pp_K$ in $F$: for any finite extension $\hat F/F$ and any factor $\pp_{\hat F}$ of $\pp_F$ we have that $e(\pp_{\hat F}/\pp_F)$ and $f(\pp_{\hat F}/\pp_F)$ are prime to $q$. Under the stronger assumption of uniformly absolutely bounded ramification, we can assume that $e(\pp_{\hat F}/\pp_F)=1$.
    
    By the Weak Approximation Theorem (see \cite[Theorem II.1.1]{Lang94}), there exist $(a,b) \in F^2$ such that for any factor $\pp_F$ of $\pp_K$ in $F$ we have that 
    \begin{enumerate}
    \item $\ord_{\pp_{F}}b=-1$;
    \item $a \equiv 1$ modulo any prime occurring in the divisor of $b$ not lying above $\pp_K$;
    \item $a$ is a unit at $\pp_F$ and is not a $q$-th power modulo $\pp_{F}$.
\end{enumerate}
The lemma follows from observing these conditions as we go up from $F$ to ${\bf K}$.  

Let $\hat F/F$ be a finite extension inside ${\bf K}$ and let $\pp_{\hat F}$ be any factor of any $\pp_F$ in $\hat F$.  
By the definition of $F$, we see that 
$$
    \ord_{\pp_{\hat F}}b=e(\pp_{\hat F}/\pp_F)\ord_{\pp_{F}}b =-e(\pp_{\hat F}/\pp_F)
$$
is coprime to $q$, and is actually equal to $-1$ under the stronger condition of uniformly absolutely bounded ramification.
The equivalence $a \equiv 1$ modulo any prime occurring in the divisor of $b$ not lying above $\pp_K$ holds automatically.  Further, since $[\kappa_{\pp_{\hat F}}:\kappa_{\pp_F}]$ is prime to $q$ and the residue class of $a$ is not a $q$-th power in $\kappa_{\pp_F}$, it is not a $q$-th power in $\kappa_{\pp_{\hat F}}$ by \cite[Theorem~VI.9.1]{Lang_Algebra}.
\end{proof}

  The main results of this section are the following two theorems.
  The first, which has more restrictive assumptions, provides an existential definition of $\OO_{\pp_K, {\bf K}}$. 
  The second theorem is more general and produces a first-order definition.
  Indeed, uniformly absolutely bounded ramification in Theorem~\ref{thm:val_ring_ex_def} means that $e(\pp_{\hat K}/\pp_K) < B$ for all finite extension $\hat K$ of $K$ inside ${\bf K}$ and all primes $\pp_{\hat K}$ of $\hat K$ over $\pp_K$ for some fixed $B$, whereas uniformly $q$-bounded ramification in Theorem~\ref{thm:val_ring_first_order_def} allows for arbitrarily large ramification indices, as long as the order at $q$ is uniformly bounded; see Definition~\ref{def:uniform_bounded}.
  
\begin{theorem}
\label{thm:val_ring_ex_def}
    If $\pp_K$ has uniformly absolutely bounded ramification and uniformly $q$-bounded relative degree in ${\bf K}$ then $\OO_{\pp_K,{\bf K}}$ has an existential definition over ${\bf K}$.
\end{theorem}
\begin{proof} 
    Take $a, b$ and $F$ from Lemma \ref{le:ab}.
Given $x \in {\bf K}$, consider the norm equation 
\begin{equation}
\label{eq:normabove}
  {\bf N}_{{\bf L}_2(\sqrt[q]{a})/{\bf L}_2}(y)=(bx^q+b^q).  
\end{equation} 
where
 ${\bf L}_1={\bf K}(\sqrt[q]{1+(bx^q+b^q)^{-1}})$ and
${\bf L}_2={\bf L}_1(\sqrt[q]{1+(a+a^{-1})b^{-1}})$.
We claim that \eqref{eq:normabove} has a solution $y$ in ${\bf L}_2(\sqrt[q]{a})$ if an only if $x\in \OO_{\pp_K, {\bf K}}$, or equivalently that $x$ is integral at all factors of $\pp_K$ in ${\bf K}$.  

First suppose that \eqref{eq:normabove} has a solution $y$ in ${\bf L}_2(\sqrt[q]{a})$.  
By applying Proposition \ref{prop:finandinfin} with $\alpha = \sqrt[q]{a}$, $z=y$, $\beta_1 = \sqrt[q]{1+(bx^q+b^q)^{-1}}$, $\beta_2 =\sqrt[q]{1+(a+a^{-1})b^{-1}}$, $\bf L = {\bf L}_2$ and $E=F(x)$ there exists a finite extension $\hat F$ of $F(x)$ such that for $L_2=\hat F(\beta_1,\beta_2) $ the norm equation 
\begin{equation}
\label{eq:normbelow}
     {\bf N}_{{L_2}(\sqrt[q]{a})/{L_2}}(y)=(bx^q+b^q)   
\end{equation}
has a solution $y \in L_2$.  
    By Lemma \ref{le:ab} and our assumption on $\pp_K$ we have that $a, b \in K$ satisfy conditions of Proposition \ref{prop:L2/L1} in $\hat F$ with respect to any factor of $\pp_K$.  
    Therefore, existence of $y \in L_2(\sqrt[q]{a})$ satisfying \eqref{eq:normbelow} implies that $x$ is integral at all factors of $\pp_K$ in $\hat F$ and therefore in $\bf K$.

Suppose now that $x \in {\bf K}$ is integral at all factors of $\pp_K$ in ${\bf K}$.  Let $\hat F=F(x)$.  Let $\beta_1, \beta_2$ be as above and let $L_2=\hat F(\beta_1,\beta_2)$.  
There exists $y \in L_2(\sqrt[q]{a})$ such that \eqref{eq:normbelow} is satisfied by Proposition~\ref{prop:L2/L1}, using the assumptions on $\pp_K, a$ and $b$, so \eqref{eq:normabove} is also satisfiable by Lemma~\ref{le:sol_down_to_sol_up}.
\end{proof}

A key ingredient in the previous proof is the fact that, due to uniform absolute boundedness of ramification, we have $\ord_{\pp_{\hat F}}b = -1$ for every finite extension $\hat F$ of $K$ inside ${\bf K}$ and every prime $\pp_{\hat F}$ of $\hat F$ lying over $\pp_K$. 
This means that $\ord_{\pp_{\hat F}}x < 0$ is equivalent to $\ord_{\pp_{\hat F}}x < \frac{q-1}{q}\ord_{\pp_{\hat F}}b$, which makes a definition of the valuation ring straightforward, given Proposition~\ref{prop:L2/L1}.

At the same time, if we only assume uniformly $q$-bounded ramification, then we only know that $\ord_{\pp_{\hat F}}b$ is coprime to $q$, and it is possible that some elements $x$ satisfy $\frac{q-1}{q}\ord_{\pp_{\hat F}}b < \ord_{\pp_{\hat F}}x < 0$, which is a case not covered by Proposition~\ref{prop:L2/L1}.
To fix this problem, we observe that if $x$ has a very small pole, then we can increase the pole by taking powers and eventually reach  $\ord_{\pp_{\hat F}}x^r < \frac{q-1}{q}\ord_{\pp_{\hat F}}b$ for some~$r$.
This observation leads to the following theorem, where the definition of the valuation ring exploits a norm equation along with a multiplicative condition.

\begin{theorem}
\label{thm:val_ring_first_order_def}
If ${\bf K}$ is uniformly $q$-bounded at $\pp_K$, then $\OO_{\pp_K,{\bf K}}$ has a first-order definition over ${\bf K}$.
\end{theorem}
\begin{proof}
Take $a,b$ and $F$ from Lemma \ref{le:ab}.
Consider the following first-order definable subsets of ${\bf K}$:
\[
    N=\{u \in {\bf K} \mid \exists y \in {\bf L}_2(\sqrt[q]{a}) 
    \text{ such that }
    {\bf N}_{{\bf L}_2(\sqrt[q]{a})/{\bf L}_2}(y)=(bu^q+b^q)\}
\]
\[
    R=\{x \in N \mid \forall u \in N: xu \in N\}
\]
where
 ${\bf L}_1={\bf K}(\sqrt[q]{1+(bx^q+b^q)^{-1}})$ and
${\bf L}_2={\bf L}_1(\sqrt[q]{1+(a+a^{-1})b^{-1}})$, as usual.
We claim that $R = \OO_{\pp_K, {\bf K}}$.

First, let $x\in R$, let $\pp_{F(x)}$ be a prime of $F(x)$ over $\pp_K$, and   suppose $\ord_{\pp_{F(x)}}x<0$.  
Then for some $r$ we have that 
\[
    \ord_{\pp_{F(x)}}x^r = r\ord_{\pp_{F(x)}}x<\frac{q-1}{q} \ord_{\pp_{F(x)}}b.
\]
Observe that we have  $x^r \in N$ by induction.  
Indeed, given $x^m\in N$ for $m\geq 1$, we deduce $x^{m+1} = x\cdot x^m \in N$ by the definition of $R$.  
Therefore, by the definition of $N$, there is an element 
$z \in {\bf K}(\sqrt[q]{\alpha})$  
such that 
${\bf N}_{{\bf L}_2(\sqrt[q]{a})/{\bf L}_2}(z)=b(x^r)^q+b^q$.

We now choose $\hat F$ by Proposition \ref{prop:finandinfin}  with $\alpha = \sqrt[q]{a}$, $\beta_1 = \sqrt[q]{1+(b(x^r)^q+b^q)^{-1}}$, and $\beta_2 =\sqrt[q]{(1+(a+a^{-1})b^{-1})}$ so that $F(x) \subset \hat F \subset {\bf K}$, and
  $$
    {\mathbf N}_{{\bf L}(\alpha)/{\bf L}}(z)
    ={\mathbf N}_{{ L(\alpha)}/{ L}}(z) = b(x^r)^q +b^q
    $$
    where $L=\hat F(\beta_1,\beta_2)$, ${\bf L}={\bf K}(\beta_1, \beta_2)$. 
By Proposition \ref{prop:L2/L1} we now have that 
\begin{equation}
    \label{eq:ineqorder}
    \ord_{\pp_{\hat F}}x^r > \frac{q-1}{q}\ord_{\pp_{\hat F}}b. 
\end{equation}
which is a contradiction. Therefore, $x\in \OO_{\pp_K, {\bf K}}$, as desired.

Conversely, suppose $\ord_{\pp_{F(x)}}x \geq 0$ for some prime $\pp_{F(x)}$ of $F(x)$ over $\pp_K$. We wish to show that $x\in R$. We can immediately see that $x\in N$ by combining Proposition~\ref{prop:L2/L1} and Lemma~\ref{le:sol_down_to_sol_up}.
Now select $u \in N$. 
By definition, there is some $z \in {\bf L}_2(\sqrt[q]{a})$ satisfying 
\[
{\mathbf N}_{{{\bf L}_2(\sqrt[q]{a})}/{ \bf L}_2}(z) = bu^q + b^q. 
\]
Applying Proposition \ref{prop:finandinfin} with $E=F(x,u)$ we conclude that there is a finite extension $\hat F$ of $F(x,u)$ with $z \in L(\sqrt[q]{a})$ and
\[
{\mathbf N}_{{ L(\alpha)}/{ L}}(z)=bu^q+b^q
\]
for  $L=\hat F(\beta_1,\beta_2)$ with $\beta_1=\sqrt[q]{1+(bu^q+b^q)^{-1}}$ and $\beta_2$, $\alpha$ as defined above.
Therefore, by Proposition \ref{prop:L2/L1} we have that 
\[
\ord_{\pp_{\hat F}}u >\frac{q-1}{q}\ord_{\pp_{\hat F}}b
\]
for any factor $\pp_{\hat F}$ of $\pp_K$.  Since $\ord_{\pp_{\hat F}}x\geq 0$, we also have that 
\[
\ord_{\pp_{\hat F}}ux >\frac{q-1}{q}\ord_{\pp_{\hat F}}b.
\]
Now let $\hat\beta_1 = \sqrt[q]{1+(b(ux)^q+b^q)^{-1}}$ and let $\beta_2$, $\alpha$ be as above. Define $\hat L = \hat F(\hat \beta_1, \beta_2)$.  Then by Proposition \ref{prop:L2/L1} there exists $\hat z \in \hat L(\sqrt[q]{a})$ such that 
\[
{\mathbf N}_{{ \hat L(\alpha)}/{ \hat L}}(\hat z)=b(xu)^q+b^q.
\]
Hence $ux \in N$ by Lemma~\ref{le:sol_down_to_sol_up}, which proves that $x\in R$ and completes the proof.
\end{proof}

In the main results of this section, we exploit norm equations where the auxiliary elements $a$ and $b$ are fixed.
In contrast, the main results of the Section~\ref{sec:q-bounded-S-ints} define rings of $\calS$-integers in ${\bf K}$, under certain conditions, by exploiting norm equations where the auxiliary elements $a$ and $b$ vary over elements of $\calV_{\calS, {\bf K}}$ and $U_{\calS, {\bf K}}$, respectively.
Consequently, we need to establish that these unit groups to be definable.
This can be deduced from the definability of valuation rings, given the following observation.

\begin{remark}
\label{rem:unit_gps_definable}
    If $\pp_K$ is a prime of $K$, then the following are definable in $\OO_{{\bf K}, \pp_K}$.
    \begin{enumerate}
        \item The unit group at $\pp_K$:  
    $$
        U_{{\bf K}, \pp_K}
            =\{u \in \OO_{\pp_K, {\bf K}} : \exists v\in \OO_{\pp_K, {\bf K}} \ uv = 1\}.
        $$
        \item The subgroup of units congruent to $1$ modulo all primes over $\pp_K$:
       $$
        \calV_{{\bf K}, \pp_K}
            =\{u \in U_{\pp_K, {\bf K}} : (u -1) \not\in U_{\pp_K, {\bf K}} \}.
        $$
    \end{enumerate}
\end{remark}

\begin{corollary}
\label{cor:unit_gps_definable}
    The following are true.
    \begin{enumerate}[(a)]
    \item    If ${\bf K}$ is uniformly $q$-bounded at a prime $\pp_{K}$ of $K$, then $U_{{\bf K}, \pp_K}$ and $\calV_{{\bf K}, \pp_K}$ are definable in ${\bf K}$.\label{cor:unit_gps_definable:gen}
        \item 
    If ${\bf K}/K$ is $q$-bounded and Galois, then $U_{\pp_K, {\bf K}}$ and $\calV_{\pp_K, {\bf K}}$ are definable in ${\bf K}$ for every prime $\pp_{K}$ of $K$. 
    \label{cor:unit_gps_definable:Galois}
    \end{enumerate}
\end{corollary}
\begin{proof}
    Part \eqref{cor:unit_gps_definable:gen} follows immediately from Theorem~\ref{thm:val_ring_first_order_def}, given Remark~\ref{rem:unit_gps_definable}.
    Then, part \eqref{cor:unit_gps_definable:gen}  implies part \eqref{cor:unit_gps_definable:Galois} by Proposition~\ref{prop:Galois_q-bounded}.
\end{proof}

\section{Defining rings of \texorpdfstring{$\calS$}{S}-integers in \texorpdfstring{$q$}{q}-bounded infinite extensions}
\label{sec:q-bounded-S-ints}

Throughout this section, $K$ is a global field of nonzero characteristic  $p\neq q$. 
For the sake of our proofs, we assume that $K$ contains a $q$-th root of unity. We remind the reader that the definability results remain true if $K$ does not contain a $q$-th root of unity; see Corollary~\ref{cor:zeta_in_K}.

Before establishing the definability of $\calS$-integers in certain $q$-bounded extensions, we prove a lemma which provides the existence of the desired auxiliary elements $a$ and $b$ used within the norm equations. 
As in Section~\ref{sec:def_val_ring}, the ultimate goal is to exploit the work of Section~\ref{sec:global} on global function fields by leveraging the property of $q$-boundedness. We remind the reader that $U_{\calS, K}$ and $\calV_{\calS, K}$ are certain groups of units; see Notation and Assumptions~\ref{notation}.

\begin{lemma}
\label{le:choose_a_b_S_integers}
    Let ${\bf K}/K$ be a $q$-bounded extension and 
    let $\calS$ be a set of primes of $K$.
    If $x\in {\bf K}$
    and  $\pp_{\bf K}$ is a prime of ${\bf K}$ which does not lie over a prime of $\calS$, 
    then there exists a finite extension $K_1$ of $K(x)$ inside ${\bf K}$ and elements 
    $a\in \calV_{\calS, K_1}$ 
    and $b\in U_{\calS, K_1}$ 
    satisfying the following properties for every finite extension $K_2$ of $K_1$ inside ${\bf K}$, 
    where we define $\pp_{K_2} = \pp_{\bf K}\cap K_2$.
    \begin{enumerate}
        \item $a$ is a unit at $\pp_{K_2}$ 
        and is not a $q$-th power modulo $\pp_{K_2}$;
        \item $\ord_{\pp_{K_2}} b \not\equiv 0\bmod q$ 
        and $\ord_{\pp_{K_2}}b < 0$;
        \item Either $\ord_{\pp_{K_2}}x \geq 0$, 
        or $\ord_{\pp_{K_2}}x < \frac{q-1}{q}\ord_{\pp_{K_2}}b$.
    \end{enumerate}
\end{lemma}
\begin{proof}
    By Proposition~\ref{prop:equiv_q_bounded_defs}.\eqref{prop:equiv_q_bounded_defs:path}, we may take $K_1$ to be a finite extension of $K(x)$ inside ${\bf K}$ satisfying the property that $\ord_q e(\pp_{K_2}/\pp_{K_1}) = 0$ and $\ord_q f(\pp_{K_2}/\pp_{K_1}) = 0$ for every finite extension ${K_2}$ of $K_1$ inside ${\bf K}$, where $\pp_{K_j} =\pp_{\bf K}\cap K_j$ for $j\in\{1,2\}$.

    By Corollary~\ref{cor:extq}, there exists a residue class $[a_0] \in \kappa_{\pp_{K_1}}$ such that $[a_0]$ is not  a $q$-th power in $\kappa_{\pp_{K_1}}$.
    By the Weak Approximation Theorem (see \cite[Theorem II.1.1]{Lang94}), there exists  $a \in \calV_{K_1,\calS}$ and $b\in U_{K_1, \calS}$ so that $[a]=[a_0]$
    and $\ord_{\pp_{K_1}}b=-1$.

    Let ${K_2}$ be a finite extension of $K_1$ inside ${\bf K}$ and define $\pp_{K_2} = \pp_{\bf K}\cap K_2$. 
    We verify the claims in order. 
    First, $[a] = [a_0]$ is a prescribed residue modulo $\pp_{K_1}$ which ensures that $a$ is a unit and not a $q$-th power in $\kappa_{\pp_{K_1}}$. If $a$ becomes a $q$-th power in $\kappa_{\pp_{K_2}}$, then $f(\pp_{K_2}/\pp_{K_1})$ is divisible by $q$ by \cite[Theorem~VI.9.1]{Lang_Algebra}, which cannot happen by the definition of $K_1$.
    Next, $\pp_{K_2}$ is a pole of $b$ because it lies over $\pp_{K_1}$, and we have
    $$
        \ord_{\pp_{K_2}}b = e(\pp_{K_2}/\pp_{K_1})\ord_{\pp_{K_1}}b\not\equiv0\bmod q
    $$ 
    by the definition of $K_1$.
    For the final claim, the case $K_1 = K_2$
    is obvious because
    $1 < \frac{q-1}{q}\ord_{\pp_{K_1}}b < 0$, so the claim in general follows because $x\in K_1$. 
\end{proof}

The following theorem is our first step towards a first-order definition of rings of $\calS$-integral functions.
We warn the reader that the formula \eqref{eq:sentence} is only a first-order formula if the unit groups $U_{{\bf K}, \calS}$ and $\calV_{{\bf K}, \calS}$ have first-order definitions themselves.
This issue was already anticipated at the end of Section~\ref{sec:def_val_ring}, and we address the problem in the corollaries that follow the theorem.

\begin{theorem}
\label{thm:infinite}
    Let ${\bf K}$ be a possibly infinite $q$-bounded algebraic extension of $K$, 
    and let $\calS$ be a nonempty finite set of valuations of some finite extension of $K$.
    For any $a,b\in {\bf K}^\times$ and $x\in {\bf K}$, define 
    ${\bf L}_1={\bf K}(\sqrt[q]{1+(bx^q+b^q)^{-1}})$, 
    ${\bf L}_3={\bf L}_1(\sqrt[q]{1+x^{-1}})$, 
    and ${\bf L}_4={\bf L}_3(\sqrt[q]{(1+(a+a^{-1})x^{-1})}$.
    Then the sentence 
    \begin{equation}
    \label{eq:sentence}
    \forall a \in \calV_{{\bf K},\calS}  \
    \forall b\in U_{{\bf K},\calS}  \
    \exists y \in {\bf L}_4(\sqrt[q]{a}) 
    : {\mathbf N}_{{\bf L}_4(\sqrt[q]{a})/{ \bf L}_4}(y)=bx^q+b^q
    \end{equation}
    is true if and only if $x$ is an $\calS$-integer.
\end{theorem}
\begin{proof}
    Let $x \in {\bf K}$ and assume that \eqref{eq:sentence} is true.
    To prove that $x$ is a $\calS$-integer, we fix a prime $\pp_{\bf K}$ of ${\bf K}$ which does not lie over $\calS$ and show that $\pp_{\bf K}$ cannot be a pole of $x$.
    Choose $a,b,K_1$ by Lemma~\ref{le:choose_a_b_S_integers} applied to ${\bf K}/K$, $\calS$, $x$ and $\pp_{\bf K}$.

    Since we assumed that \eqref{eq:sentence} is true for $x$, there exists a $y\in{\bf L}_4(\sqrt[q]{a})$ which solves the following equation: 
\begin{equation}
\label{eq:above}
  {\mathbf N}_{{\bf L}_4(\sqrt[q]{a})/{ \bf L}_4}(y)=bx^q+b^q.  
\end{equation}  
We now apply Proposition \ref{prop:finandinfin} to the extensions ${\bf K}/K_1$, ${\bf L}_4$ and the elements $y$ and $\alpha =\sqrt[q]{a}$ to obtain fields $\hat K$ and $L_4$ such that 
\begin{equation}
\label{eq:below}
    {\mathbf N}_{{L_4(\sqrt[q]{a})}/{ L}_4}(y)
    =bx^q + b^q.
\end{equation}
where the extension $L_4/\hat K$ is generated as a tower by
${L}_1={\hat K}(\sqrt[q]{1+(bx^q+b^q)^{-1}})$, 
${L}_3={L}_1(\sqrt[q]{1+x^{-1}})$, and
${L}_4={L}_3(\sqrt[q]{(1+(a+a^{-1})x^{-1})}$.
The conclusion of Lemma~\ref{le:choose_a_b_S_integers} with the choice $K_2 = \hat K$ states that $a$ and $b$ satisfy all of the necessary hypotheses of  Lemma~\ref{le:nosolution} by construction.
Therefore, $\pp_{\bf K}$ cannot lie over a pole of $x$. 

To complete the proof, assume $x \in {\bf K}$ is an $\calS$-integer,
and choose any $b \in U_{{{\bf K},\calS}}$ and $a\in \calV_{{\bf K},\calS}$.  
Consider  $K_1=K(x,a,b)$ and let $L_1, L_3$ and $L_4$ be defined as $\bf {L}_1, \bf {L}_3$ and $\bf {L}_4$ respectively with $\bf K$ replaced by $K_1$. 
By Lemma \ref{le:existsolutions}, equation \eqref{eq:below} has a solution $y \in L_4(\sqrt[q]{a})$, so \eqref{eq:above} has a solution in ${\bf L}_4(\sqrt[q]{a})$ by Lemma~\ref{le:sol_down_to_sol_up}
\end{proof}

\begin{corollary}
\label{cor:fo}
    In the notation above, if the sets $\calV_{{\bf K},\calS}$ and $U_{{\bf K},\calS}$ are diophantine subsets of ${\bf K}$, then the ring of $\calS$-integers has a first-order definition over ${\bf K}$.
\end{corollary}
\begin{proof}
    By Lemma \ref{le:norm_as_diophantine}, the set
\[
    A_{\bf K} := \{(a,b,x)\in {\bf K}^3 \mid \exists y \in {\bf L}_4, \ {\mathbf N}_{{L_4(\sqrt[q]{a})}/{L}_4}(y)=bx^q+b^q \}
\]
is diophantine over ${\bf K}$.
Hence, by Lemma \ref{le:fo} we have that the subset of ${\bf K}$ defined by \eqref{eq:sentence} is first-order definable over ${\bf K}$.
\end{proof}

\begin{corollary}    
\label{cor:S-ints_def}
 Let $\calS$ be a finite set of primes of $K$. Assume either that ${\bf K}/K$ is Galois, or that every prime in $\calS$ is uniformly $q$-bounded in~${\bf K}$. Then the ring of $\calS$-integers $\OO_{{\bf K}, \calS}$ is definable in ${\bf K}$.

\end{corollary}

\begin{proof}
     Corollary~\ref{cor:unit_gps_definable} proves that $U_{{\bf K}, \calS}$ and $\calV_{{\bf K}, \calS}$ are definable in ${\bf K}$ in these cases, so Theorem~\ref{thm:infinite} provides a first-order definition of $\calS$-integers.
\end{proof}
\begin{corollary}
\label{cor:infmany}
    Let ${\bf K}/\F_p(t)$  be a $q$-bounded extension.  Then there are infinitely many non-constant $u$ such that the integral closure $\OO_{\mathbf{K},u}$ of $\F_p[u]$ in ${\bf K}$ is definable over ${\bf K}$.
\end{corollary}
\begin{proof}
Let $\pp_{\F_p(t)}$ be any prime of $\F_p(t)$.  By Corollary \ref{cor:existunifqbound}, there are infinitely many pairs $(K,\pp_K)$ where $K \subset {\bf K}$ is finite extension of $\F_p(t)$ and the $K$-prime $\pp_K$ is uniformly bounded in ${\bf K}$.  By the Strong Approximation Theorem (see \cite[\S15]{Cassels_GlobalFields}), there exists $u \in K$ such that $u$ has a pole at $\pp_K$ only.   Let $\calS=\{\pp_K\}$.  Then the integral closure of $\OO_{{\bf K},\calS}$ is definable over ${\bf K}$ by Corollary~\ref{cor:S-ints_def}.  Observe that $\OO_{K,\calS}$ is the integral closure of $\F_p[u]$ in $K$.  Therefore, therefore the integral closure of $\F_p[u]$ is definable over ${\bf K}$.
\end{proof}

\begin{remark}
    In this paper we generically used only one prime $q$ to describe ``boundedness''.  However, as described in \cite{Shlapentokh18}, one can use different rational primes for different purposes.  For example in a $q$-bounded field, a prime of $\F_p(t)$ can be uniformly $q'$-bounded for some prime $q' \ne q$.  So one would use $q'$ in the definition of the integral closure of a valuation ring.  
    
    More generally, in some fields the primes of $\F_p(t)$ can be partitioned into a finite collection of sets with each set of primes $q$-bounded with a different $q$.  There are examples of such infinite extensions in \cite{Shlapentokh18} and a similar construction will work in some infinite extension of $\F_p(t)$.
\end{remark}

\begin{theorem}
\label{thm:empty}
    If ${\bf K}$ is a $q$-bounded algebraic extension of $\F_p(t)$, then the set of all elements of ${\bf K}$ algebraic over $\F_p$ has a first-order definition over ${\bf K}$.
\end{theorem}
\begin{proof}
Consider the following first-order sentence:
\begin{equation}
    \label{eq:sentence_constants}
    \forall a \  \forall b \ \exists y \in {\bf L}_4(\sqrt[q]{a}) : {\mathbf N}_{{\bf L}_4(\sqrt[q]{a})/{ \bf L}_4}(y)=bx^q+b^q
    \end{equation}
Now the argument similar to the one used to prove Theorem \ref{thm:infinite} shows that this sentence is true for some $x \in {\bf K}$ if and only if $x$ has no poles at any valuation of $\F_p(t,x)$.  Hence, the sentence is true of $x \in {\bf K}$ if and only if $x$ is a constant.  Now, by the same argument used in the proof of Corollary \ref{cor:fo}, we have that \eqref{eq:sentence_constants} is a first-order sentence in the language of rings.
\end{proof}

\section{First-order undecidability of rings of \texorpdfstring{$\calS$}{S}-integers} 
\label{sec:decidability}
In this section we describe some results concerning undecidability and infinite extensions of global fields.
We write $\OO_{{\bf K}}$ for the integral closure of $\F_p[t]$ in ${\bf K}$, which is equivalent to taking $\calS$ to be the set containing the pole of $t$ in the notation of the previous section.
Recall that the constant subfield of an extension ${\bf K}/\F_p(t)$ is the algebraic closure of $\F_p$ inside~${\bf K}$.
\subsection{The case of an infinite constant field}
\begin{proposition}
\label{prop:polyring}
    Let ${\bf K}$ be any extension of $\F_p(t)$
    and let $\bf k$ be the field of constants of ${\bf K}$. 
    If ${\bf k}$  is infinite and  first-order definable in $\OO_{{\bf K}}$, then ${\bf k}[t]$ is first-order definable in $\OO_{{\bf K}}$. 
\end{proposition}
\begin{proof}
    Let $x \in \OO_{{\bf K}}$ and suppose that for every $c\in {\bf k}$ there exists $ b \in {\bf k}$ such that $x -b\equiv 0 \bmod (t-c)$.  
    Let $\hat {\bf K}$ be the normal closure of ${\bf K}$ over ${\bf k}(t)$.  
    Let $\OO_{\hat {\bf K}}$ be the integral closure of $\OO_{{\bf K}}$ in $\hat {\bf K}$.  
    Then the equivalence holds in $\OO_{\hat {\bf K}}$ too. 
    Let $\hat x$ be any conjugate of $x$ over ${\bf k}(t)$.  
    Then $\hat x -b\equiv 0 \bmod (t-c)$ and $\hat x -x \equiv 0 \bmod (t-c)$.  
    Since the last equivalence must hold for infinitely many $c$, we conclude that $x=\hat x$ or, alternatively, $x \in {\bf k}(t) \cap \OO_{{\bf K}}={\bf k}[t]$.
\end{proof}

\begin{corollary}
    \label{cor:dec_integral_infinite_constants}
    If ${\bf K}$ is a $q$-bounded algebraic extension of $\F_p(t)$ with an infinite fields of constants ${\bf k}$, then $\F_p[t]$ is definable in $\OO_{\bf K}$ and the first-order theory of $\OO_{\bf K}$ is undecidable.
\end{corollary}
\begin{proof}
    By Theorem~\ref{thm:empty}, the field of constants is definable in ${\bf K}$, and therefore also in $\OO_{\bf K}$. 
    Thus, ${\bf k}[t]$ is definable in $\OO_{\bf K}$ by Proposition~\ref{prop:polyring}.
    By a result of the first author \cite{Shlapentokh93}*{Theorem 4.17}, we have that $\F_p[t]$ is definable in ${\bf k}[t]$, so $\F_p[t]$ is definable in $\OO_{\bf K}$.  
    By a result of Denef (\cite{Denef79}), the theory of $\F_p[t]$ is undecidable.  
    Hence the theory of $\OO_{\bf K}$ is undecidable.
\end{proof}

\begin{corollary}
\label{cor:dec_field_infinite_constants_general}
    Let ${\bf K}$ be a $q$-bounded algebraic extension of $\F_p(t)$ with an infinite constant field.  Then the first-order theory of ${\bf K}$ is undecidable.
\end{corollary}
\begin{proof}
    This is a combination of  Corollary~\ref{cor:infmany} and Corollary~\ref{cor:dec_integral_infinite_constants}.
\end{proof}

\subsection{The case of finitely many factors}
\label{subsec:fin_factors}
In this section we consider a class of algebraic extensions  ${\bf K}$ of $\F_p(t)$ with constant field ${\bf k}$ enjoying the following property.  
There exists $t \in {\bf K} \setminus {\bf k}$  and a prime $\pp_{{\bf k}(t)}$ of ${\bf k}(t)$ with finitely many distinct factors in ${\bf K}$.   
The fact that $\pp_{{\bf k}(t)}$ has finitely many distinct factors only in ${\bf K}$ implies that there exists a finite extension $K$ of ${\bf k}(t)$ and a prime $\pp_K$ of $K$ such that $\pp_K$ has only one prime $\pp_{\bf K}$ above it in ${\bf K}$. Let $\calS_{\bf K}=\{\pp_{\bf K}\}$.  
Then by \cite{Shlapentokh92}*{Theorem~5.1} the ring $\calO_{{\bf K},\calS_{\bf K}}$ is existentially undecidable in the language $(0,1, \times, +, s)$, where $s$ is a parameter representing either a non-constant element or equivalently a non-unit of the ring (whether or not the field of constants is infinite).  

If we use the full first-order language, then we can say that an element of the ring is not a unit and therefore not a constant.  Hence, the first-order theory of such a ring is undecidable in the ring language without parameters. The fields to which \cite{Shlapentokh92}*{Theorem~5.1} applies include the towers of rational function fields from the undecidability result in \cite{MRSU24}.  In particular, in the undecidability result in \cite{MRSU24}, no restriction on $r$ or characteristic is necessary. The following summarizes our decidability result in the case of a prime with finitely many factors.

\begin{theorem}
    Let ${\bf K}$ be an $q$-bounded extension of a global function field $K$ and assume there is a prime $\pp_{K}$ of $K$ with only finitely many factors in ${\bf K}$. Then the first order theory of ${\bf K}$ in the language $\{0,1,+,\times\}$ is undecidable.
\end{theorem}
\begin{proof}
This follows from the discussion above. Namely, letting $\calS_{\bf K}$ be the primes of ${\bf K}$ over $\pp_K$, we observe that the first-order theory of the ring of $\calS_{\bf K}$-integral functions $\OO_{K, \calS_{\bf K}}$ is undecidable by \cite{Shlapentokh92}*{Theorem~5.1} and $\OO_{K, \calS_{\bf K}}$ is definable in ${\bf K}$ by Corollary~\ref{cor:S-ints_def}.
\end{proof}

In \cite{Shlapentokh92}, there is no discussion of how to construct fields admitting a prime with finitely many factors.  To conclude the paper, we show that such a construction is easy, concentrating on the cases where the algebraic closure of $\F_p$ in ${\bf K}$ is finite. We use a construction designed to produce a prime inert in an infinite extension. In comparison, a tower of fields $\F_p(t^{1/m_i})$ with $m_i | m_{i+1}$ produces a totally ramified prime in the infinite extension.  In either case the result of \cite{Shlapentokh92}*{Theorem~5.1} applies, since the relevant condition is for a prime of $\F_p(t)$ to have only finitely many distinct ideals above it in the infinite extension. \\

We start with two preliminary lemmas.

\begin{lemma}
\label{le:inert2}
  Let $K$ be a global function field and let $L$ be a finite extension of $K$.  Let $\pp_K$ be a prime of $K$, let $\alpha$ be a generator of $L/K$ and assume that $\alpha$ is integral with respect to $\pp_K$.  Let $g(X)\in K[X]$ be the monic irreducible polynomial of $\alpha$ over $K$.  If $g(X)$ remains irreducible in the residue field of $\pp_K$, then the prime $\pp_K$ is inert in the extension $L/K$.
\end{lemma}
\begin{proof}
    Let $\pp_L$ be a prime of $L$ over $\pp_K$.  The element $\alpha$ reduces to a root of $g(X)$ in the residue field $\kappa_{\pp_L}$, so we have 
    $$
     f(\pp_L/\pp_K) = [\kappa_{\pp_L} : \kappa_{\pp_K}] = \deg(g) = [L : K]
     $$ 
     because $g(X)$ is irreducible modulo $\pp_K$. Therefore, $\pp_K$ is inert in $L$ by definition.
\end{proof}

\begin{lemma}
\label{le:totram}
    Let $K$ be a global function field and let $P(X)=A_0+A_1X+\ldots X^r \in \F_p(t)[X]$ be such that there exists a prime $\qq_K$ such that $\ord_{\qq_K}A_0=-1$ and $\ord_{\qq_K}A_i \geq 0$ for $i>0$.  Then
    \begin{enumerate}
        \item $P(X)$ is irreducible.
        \item If $L$ is an extension of $K$ generated by a root of $P(X)$, then $\qq_K$ is completely ramified in the extension $L/K$.
    \end{enumerate}
\end{lemma}
\begin{proof}
    Let $z \in \bar K$ be a root of $P(X)$.  Then
    \[
    \sum_{i=1}^rA_iz^i=-A_0
    \]
    Let $\qq_{K(z)}$ be a prime above $\qq_K$ in $K(z)$.  Then, $\ord_{\qq_{K(z)}}z<0$ and 
    \[
    \ord_{\qq_{K(z)}}(A_0)=\ord_{\qq_{K(z)}} \sum_{i=1}^rA_iz^i=\ord_{\qq_{K(z)}}z^r=r\ord_{\qq_{K(z)}}z.
    \]
    Hence, $\ord_{\qq_{K(z)}}(A_0)\equiv 0 \bmod r$ while $\ord_{\qq_K}(A_0)=-1$.  Therefore, $e(\qq_{K(z)}/\qq_K)\equiv 0 \bmod r$ and hence the ramification is equal to $r$ implying that $P(X)$ is irreducible and $\qq_K$ is totally ramified in the extension $K(z)/K$.
\end{proof}

We now construct the desired fields extensions.
Let $\pp_{\F_p(t)}$ be the zero of $t$ in $\F_p(t)$ denoted in the future by $\pp_0$.  Let 
\[
p(t)=\sum_{i=0}^na_it^i\in \F_p(t)
\]
be  monic and irreducible over $\F_p$.  Consider now a polynomial 
\[
P(t,X)=\frac{a_0}{t+1}+\sum_{i=1}^na_i(t+X)^i \in \F_p(t)[X].  
\]
Modulo $\pp_0$, this polynomial is irreducible over $\F_p$ and therefore $P(t,X)$ is irreducible over $\F_p(t)$. Let $L_1$ be the extension of $\F_p(t)$ generated by a root of $P(t,X)$. By Lemma \ref{le:inert2}, we have that $\pp_0$ is inert in this extension.  Further, since the prime corresponding to $t+1$ is totally ramified in the extension $L_1/\F_p(t)$ by Lemma \ref{le:totram}, we have that there is no constant field extension between $\F_p(t)$ and $L_1$.

Assume inductively that we have constructed an extension $L_m$ of $\F_p(t)$ such that $\pp_0$ is inert in this extension and the constant field of $L_m$ is $\F_p$.  Let $\pp_m$ be the prime of $L_m$ above $\pp_0$.  Let $p_m(X)=\sum_{i=0}^{n_m}a_{i,m}X^i\in \F_p[X]$ be monic and irreducible over the residue field of $\pp_m$.  

Observe that there are only finitely many primes ramified in the extension $L_m/\F_p(t)$.  Therefore, there exists a positive integer $\ell$ such that any $\F_p(t)$-prime corresponding to a monic irreducible polynomial of degree higher than $\ell$ does not have any ramified factors in $L_m$. 
Let $q_m(t)$ be such that $q_m(t) \in \F_p[t]$, $\deg q_m(t) > \ell$, and $q_m(t)$ is monic and irreducible over $\F_p[t]$.   Then if $\qq_m$ is a prime of $L_m$ lying above the $\F_p(t)$-prime corresponding to the polynomial $q_m(t)$ we have that $\ord_{\qq_m}q_m(t)=1$. Let 
\[
    P_m(t,X)=\frac{a_{0,m}q_m(0)}{q_m(t)}+\sum_{i=1}^{n_m}a_{i,m}(t+X)^i \in \F_p(t)[X].
\]
Then $P_m(t,X) \equiv p_m(X)$ modulo $\pp_m$, where $p_m(X)$ is irreducible modulo $\pp_m$.  Therefore, $P_m(t,X)$ is irreducible over $L_m$.  Let $L_{m+1}$ be generated by a root of $P_m(t,X)$.  Then by Lemma \ref{le:inert2} we have that $\pp_m$ is inert in the extension $L_{m+1}/L_m$.  

At the same time $\qq_m$ is completely ramified in the extension $L_{m+1}/L_m$ and therefore there is no constant field extension from $L_m$ to $L_{m+1}$.  Thus, the constant field of $L_{m+1}$ is still $\F_p$.  

Let ${\bf K}=\bigcup_{i=0}^{\infty}L_i$.  Then the constant field of ${\bf K}$ is $\F_p$ and $\pp_0$ has a unique factor in ${\bf K}$, as desired, so the construction is finished.

\section{Defining arbitrary polynomial rings over \texorpdfstring{${\bf K}$}{K}}
\label{sec:def_poly_rings}
Our goal in this section is to show that if we can define one $\F_p$-polynomial ring over ${\bf K}$, then we can define all of them.  In other words, we assume that we are given two non-constant elements $u$ and $w$ of $\bf K$ such that $\F_p[u]$ has a first-order definition (with parameters) over ${\bf K}$ and show that the same is true of $\F_p[w]$. We deduce that all $\F_p$-polynomials rings are definable when ${\bf K}$ is a  $q$-bounded extension with infinite constant field.

Our proof exploits a result of J. Demeyer (\cite{Demeyer07}) which asserts that any c.e. subset of $\F_p[u]$ has an existential definition over $\F_p[u]$.  Before we state and prove the theorem, we first discuss the presentation of countable fields and rings and the notion of computable and c.e. subsets of these algebraic objects.
To  discuss computability of countable algebraic objects we map these objects into $\Z_{\geq 0}$ and translate binary operations over the object by functions from $\Z_{\geq 0}^2$ to $\Z_{\geq 0}$.  We call a field or ring computable if there exists a map as described above translating the ring operations by total computable functions.  

By a famous theorem of Rabin \cite{Rabin}, every countable field with a splitting algorithm (i.e. an algorithm to factor polynomials into irreducible factors) has a computable algebraic closure such that the domain of the field in question as a subset of $\Z_{\geq 0}$ is computable.  In our case the field in question is $\F_p(t)$, where $\F_p$ is a field of $p$ elements and $t$ is transcendental over $\F_p$.  It is well-known that this field has a splitting algorithm; see \cite{FJ23}.  Therefore, there exists a computable presentation of $\overline{\F_p(t)}$ such that the field $\F_p(t)$ is a computable subset of it.  We will identify our field ${\bf K}$ with a subfield of $\overline{\F_p(t)}$ presented as above.  Note that we do not assume that ${\bf K}$ is computable under this or any other presentation.

Given any transcendental element $u$ of $\bf K \subset \overline{\F_p(t)}$,  the field $\F_p(u)$ is a computable subset of $\overline{\F_p(t)}$.  Indeed, the domain of the field $\F_p(u)$ is certainly c.e. since we can list all rational functions in $u$ with coefficients in $\F_p$. Therefore, we can list all polynomials in some variable $T$ over $\F_p(u)$.

Since $\overline{\F_p(t)}$ is of transcendence degree 1, every element $x$ of $\overline{\F_p(t)}$ satisfies some monic polynomial $P(T)$ over $\F_p(u)$ which we can find by going through a list of all monic polynomials with coefficients in $\F_p(u)$.  Once we find a polynomial $P(T)$ such that $P(x)=0$, we can use the splitting algorithm over $\F_p(u)$ to factor $P(T)$ into  irreducible factors over $\F_p(u)$ and determine whether one of the factors is of the form $(T-x)$.  If such a linear factor exists, then $x \in \F_p(u)$ and it is not in $\F_p(u)$ otherwise.

If $w \in \overline{\F_p(t)}\setminus \F_p(u)$, we can determine the degree of $w$ over $\F_p(u)$.  Denoting this degree by $n$  we can represent every element of $\F_p(u,w)$ in  the form 
\[
\sum_{i=0}^{n-1}\frac{a_i(u)}{b(u)}w^i,
\]
where $a_i(u), b(u) \in \F_p(u)$. Since $\F_p(w)$ is also a computable field, we now have an algorithm to determine which $n+1$-tuples $(a_0(u),\ldots, a_{n-1}(u), b(u)) \in \F_p(u)^{n+1}$ correspond to elements of $\F_p(w)$ when used in the linear combination above.  Thus, the set  
\begin{equation}
    \label{eq:A}
    A=\{(a_0(u),\ldots, a_{n-1}(u), b(u)) \in \F_p[u]^{n+1} \mid \sum_{i=0}^{n-1}\frac{a_i(u)}{b(u)}w^i \in \F_p(w)\}
\end{equation}
is a computable subset of $\F_p(u)$.  Further, the set
\begin{equation}
    \label{eq:B}
    B=\{(a_0(u),\ldots, a_{n-1}(u), b(u)) \in \F_p[u]^{n+1} \mid \sum_{i=0}^{n-1}\frac{a_i(u)}{b(u)}w^i \in \F_p[w]\}
\end{equation}
is also computable.  To see that $B$ is computable note that we can generate a list of elements of $\F_p[w]$ and the list of elements of $\F_p(w)\setminus \F_p[w]$ using the splitting algorithm over $\F_p$.  Using this discussion, we can now prove the following theorem.

\begin{theorem}
\label{thm:def}
 Let ${\bf K}$ be an algebraic extension of $\F_p(t)$ and assume that for some non-constant $u$ we have a definition (with parameters) of $\F_p[u]$ over ${\bf K}$.   Then for any non-constant $w$ we have a definition (with parameters) of $\F_p[w]$.
\end{theorem}
\begin{proof}
    Let $\psi(T)$ be a first-order definition of $\F_p[u]$ over ${\bf K}$.
    Given the discussion of computability above, the set $B$ defined in~\eqref{eq:B} has a diophantine definition $p(a_0,\ldots, a_{n-1},b, Y_1,\ldots,Y_m)$ over $\F_p[u]$ by the result of J. Demeyer (\cite{Demeyer07}).
   Thus the following formula defines $\F_p[w]$: An element $y \in {\bf K}$ is contained in $\F_p[w]$ if and only if $\exists a_0,\ldots,a_{n-1}, b, Y_1,\ldots, Y_m \in {\bf K}$ so that
\[
\bigwedge_{i=0}^{n} \psi(a_i) \land \psi(b) \land \bigwedge_{i=1}^m\psi(Y_i) \land p(a_1,\ldots, a_{n-1},b, Y_1,\ldots,Y_m)=0 \land y=\sum_{i=0}^{n-1}\frac{a_i(u)}{b(u)}w^i. 
\]
\end{proof}

We can now apply Theorem \ref{thm:def} to the case of a $q$-bounded ${\bf K}$ with an infinite field of constants to obtain the following corollary.

\begin{corollary}
\label{cor:any_poly_ring_q_bounded}
Let ${\bf K}$ be a $q$-bounded algebraic extension of $\F_p(t)$ with an infinite constant field.  Let $w$ be any non-constant element of ${\bf K}$.  Then $\F_p[w]$ is first-order definable over ${\bf K}$ with parameters. 
\end{corollary} 
\begin{proof}
By Corollary \ref{cor:infmany} there exists a non-constant element $u \in {\bf K}$ such that $\OO_{{\bf K},u}$, the integral closure of $\F_p[u]$ in ${\bf K}$, is definable over ${\bf K}$. 
By Corollary~\ref{cor:dec_integral_infinite_constants}, if ${\bf K}$ has an infinite constant field, then $\F_p[u]$ is definable over $\OO_{{\bf K},u}$, and Theorem~\ref{thm:def}  applies.
\end{proof}
\begin{remark}
From an examination of the statements of Lemma~\ref{le:ab} and Theorem~\ref{thm:val_ring_first_order_def} it is easy to see that we need two parameters to define valuation rings for valuations occurring as poles of $u$ and we need a parameter for $u$ to define $\F_p[u]$.   Finally, we need a parameter for $w$ to define $\F_p[w]$.   
\end{remark}

\section{Appendix}
\subsection{Valuations of an infinite algebraic extension of \texorpdfstring{$\F_p(t)$}{Fp(t)}}
\label{ssec:valuations}
The only valuations of algebraic extensions of finite fields are trivial by virtue of the fact that all nonzero elements in such extensions have finite (multiplicative) order.
Let ${\bf K}$ be an infinite algebraic extension of $\F_p(t)$ and let ${\bf v}$ be a valuation of ${\bf K}$.  If $K \subset {\bf K}$ and $K/\F_p(t)$ is finite, then ${\bf v}_{|K}$ must be a discrete valuation $v_K$ since $K$ only has discrete valuations.  

If ${\bf \pp}$ is the valuation ideal of ${\bf v}$, then ${\bf \pp}$ contains the valuation ideal $\pp_K$ of $K$ corresponding to $v_K$ for any finite extension $K$ of $\F_p(t)$ contained in ${\bf K}$.  
At the same time, $v_K$ can have infinitely many extensions ${\bf v}$ to ${\bf K}$.  It is possible for $e({\bf v}/v)=\infty$ though the ramification can also be finite.  Similarly, $f({\bf v}/v)$ can be finite or infinite.  
The completion of ${\bf K}$ at different valuations ${\bf v}$ extending the same valuation $v$ of $\F_p(t)$ can correspond to different embeddings of ${\bf K}$ into $\overline{\F_p(t)_v}$.

In the definition of $q$-boundedness (Definition~\ref{def:q-bounded_tower_of_completions}), we refer to a tower of completions ${\bf K}_{T,\bf v}$ instead of referring to the completion ${\bf K}_{\bf v}$ itself.
This technical detail arises because the completion ${\bf K}_{\bf v}$ can be a transcendental extension of $K_{v_K}$ under certain conditions; see the proposition below.
In contrast, ${\bf K}_{T,\bf v}$ is always an algebraic extension of $K_{v_K}$ by design.
\begin{proposition}
If ${\bf  v}$ is unramified over $v_K$ and $\kappa_{\bf v}/\kappa_{v_K}$ is an infinite (algebraic) extension, then ${\bf K}_{\bf v}$ is a transcendental extension of $K_{v_K}$.
In particular, ${\bf K}_{\bf v}\supsetneq {\bf K}_{T,\bf v}$.
\end{proposition}
\begin{proof}
Let $\{\alpha_0, \alpha_1,\dots\} \subseteq {\bf K}$ be a subset whose residue classes generate $\kappa_{\bf v}$ over $\kappa_{v_K}$, and   
let $y \in K^\times$ be an element such that ${\bf v}(y)  > 0$.
Consider the element $x =\sum_{i=0}^\infty \alpha_i y^i$, observing that $x \in {\bf K}_{\bf v}$ because the series converges in the ${\bf v}$-adic topology. 
We will prove that $x$ is transcendental over $K_{v_K}$.

Define $L_0=K_{v_K}(x)$ and $L_i=K_{v_K}(x, \alpha_0,\ldots, \alpha_{i-1})$ for $i \geq 1$, and let $v_{L_i}$ be the restriction of ${\bf v}$ to $L_i$.
Observe that $e(v_{L_i}/v_{L_0})=1$ for all $i$ because the extension ${\bf v}/v_K$ is unramified by assumption.
By the definition of $x$, we see $[x] = [\alpha_0]$ in $\kappa_{\bf v}$. Therefore, $\kappa_{v_{L_0}} = \kappa_{v_{L_0}}(\alpha_0)$, hence $f(v_{L_1}/v_{L_0}) = 1$.
This proves that 
$$[L_1:L_0]=e(v_{L_1}/v_{L_0})f(v_{L_1}/v_{L_0})=1.$$ 
In particular, we have $\alpha_0\in L_0$.
Replacing $x$ by $\frac{x-\alpha_0}{y}$ and continuing by induction, we find that $\alpha_i\in L_0$ for all $i$.
Therefore, because $\kappa_{v_{L_0}} = \kappa_{\bf v} = \kappa_{v_K}(\{\alpha_i\})$ is an infinite extension of $\kappa_{v_K}$, we see that $L_0 = K_{v_K}(x)$ must be an infinite extension  $K_{v_K}$.
Consequently, $x$ is a transcendental element over $K_{v_K}$, as claimed.
\end{proof}

 \bibliographystyle{alpha}
 \bibliography{references}
\end{document}